\documentclass[preprint,11pt,authoryear]{elsarticle}

\usepackage[left=3.5cm, right=3.5cm, top=3.5cm, bottom=3.5cm]{geometry}

\usepackage{amsmath,amssymb}
\usepackage{hyperref}
\usepackage{appendix}
\usepackage{ascmac}
\usepackage{bm}
\usepackage{amsthm}
\usepackage{here}
\usepackage{natbib}
\usepackage{mathtools}

\usepackage{algorithmic}
\usepackage{algorithm}

\usepackage{mdframed}

\usepackage{booktabs}
\usepackage{multirow}

\usepackage{adjustbox}

\theoremstyle{definition}
\newtheorem{theorem}{Theorem}[section]
\newtheorem{proposition}[theorem]{Proposition}
\newtheorem{lemma}[theorem]{Lemma}
\newtheorem{corollary}[theorem]{Corollary}
\newtheorem{definition}[theorem]{Definition}
\newtheorem{remark}[theorem]{Remark}

\newcommand{\E}{\mathbb{E}}
\newcommand{\R}{\mathbb{R}}

\begin{document}

    \begin{frontmatter}



        \title{Precomputable Trade-off Between Error and Breakpoints in Piecewise Linearization for First-Order Loss Functions}


        \author[inst1]{Yotaro Takazawa}
        \ead{takazawa@meiji.ac.jp}
        \affiliation[inst1]{organization={School of Business Administration, Meiji University},
            addressline={1-9-1, Eifuku},
            city={Suginami-ku},
            state={Tokyo},
            country={Japan}}

        \begin{abstract}
            Stochastic optimization often involves calculating the expected value of a first-order max or min function, known as a first-order loss function.
            In this context, loss functions are frequently approximated using piecewise linear functions.
            Determining the approximation error and the number of breakpoints (segments) becomes a critical issue during this approximation.
            This is due to a trade-off: increasing the number of breakpoints reduces the error but also increases the computational complexity of the embedded model.
            As this trade-off is unclear in advance, preliminary experiments are often required to determine these values.

            The objective of this study is to approximate the trade-off between error and breakpoints in piecewise linearization for first-order loss functions.
            To achieve this goal, we derive an upper bound on the minimum number of breakpoints required to achieve a given absolute error.
            This upper bound can be easily precomputed once the approximation intervals and error are determined, and serves as a guideline for the trade-off between error and breakpoints.
            Furthermore, we propose efficient algorithms to obtain a piecewise linear approximation with a number of breakpoints below the derived upper bound.

        \end{abstract}


    \end{frontmatter}



    \section{Introduction}
    Stochastic optimization often involves the calculation of an expected value of a max (min) function consisting of a decision variable and a random variable.
    One of the most fundamental functions among them can be expressed as a univariate function \( \ell_X: \mathbb{R} \rightarrow \mathbb{R} \),
    \begin{equation}
        \label{expectation_function}
        \ell_X(s) = \E [\max(a_1s+b_1X+c_1, a_2s+b_2X+c_2)],
    \end{equation}
    where \( X \) is a given one-dimensional random variable and \( a_i, b_i, c_i \in \mathbb{R} \) (\( i \in \{1, 2\} \)).
    This function is used in many areas, especially in inventory control.
    For example, let \( X \) be an uncertain demand for some product \( P \), and \( s \) be the ordering quantity for \( P \).
    Then \( \E[\max(X-s, 0)] \) and \( \E [\max(s- X, 0)] \) are considered as shortage and inventory costs in inventory control.
    These functions are known as first-order loss functions \citep{snyder2019fundamentals} and are used in inventory control and other applications \citep{rossi2014piecewise}.
    Thus, we call \( \ell_X \) a general (first-order) loss function, as it is a generalization of them. As another example, $\E [ \min(s, X)]$ is regarded as the expected number of units of product P sold, which is often used in a profit function in newsvendor models.

    While a general first-order loss function is often embedded in mixed-integer linear programming (MILP) models, calculating these values is a challenging task for the following reasons.
    When a target random variable \( X \) is continuous, these expectation functions are often nonlinear and thus difficult to embed in MILP.
    On the other hand, when \( X \) is discrete, these expectation functions can be directly embedded in MILP.
    However, when $X$'s support has a large cardinality or is infinite, solving optimization problems becomes challenging.

    For the above reasons, loss functions are often approximated by a piecewise linear function, which is a tractable format for MILP.
    In a piecewise linear approximation, we choose \( n \) points in ascending order, known as breakpoints.
    We approximate the loss function using linear segments connecting each breakpoint to its corresponding function value.

    Determining the appropriate parameters, such as the acceptable error and the number of breakpoints for piecewise linear functions, is a crucial aspect of this approach.
    While increasing the number of breakpoints can decrease the approximation error, it also increases the computational complexity of MILP.
    Therefore, we must set a careful balance when choosing these parameters.
    Often, the appropriate parameters are determined through preliminary and sometimes heavy numerical experiments, as the theoretical relationship between these parameters is not always clear.

    This study aims to understand the trade-off between error and breakpoints in a piecewise linear approximation for general first-order loss functions.
    To achieve this objective, we derive a tight upper bound for the minimum number of breakpoints needed to achieve a given error.
    We also propose an efficient method for constructing a piecewise linear function with a number of breakpoints below this upper bound.
    As a result, we enable the determination of appropriate levels of error and breakpoints, using this upper bound as a guideline.

    Specifically, we have obtained the following results in this study for a general loss function:
    \begin{enumerate}
        \item Given an approximation error \( \epsilon > 0 \), we propose algorithms to make a piecewise linear function
        such that the maximum absolute error is within \( \epsilon \)
        and the breakpoints are bounded by \( M\sqrt{\frac{W}{\epsilon}} \),
        where \( M \ (\leq 1) \) is a parameter dependent on the setting and \( W \) is the width of the approximation interval. Among the proposed algorithms, one guarantees minimality in the number of breakpoints,
        while the others, although not guaranteeing minimality, offer the same upper bound on the number of breakpoints. We also demonstrate that this upper bound on the number of breakpoints is tight from a certain perspective.

        \item Through computational experiments, we compare the actual number of breakpoints generated by our proposed algorithms
        with the derived bounds across various distributions.
        In many cases, we find that the minimal number of breakpoints can be approximated by \( \frac{1}{2\sqrt{2}}\sqrt{\frac{W}{\epsilon}} \).
    \end{enumerate}

    We review related approaches for piecewise linear approximation, specifically focusing on methods that provide theoretical guarantees in terms of approximation error or the number of breakpoints.
    These can be broadly divided into two categories based on whether they fix the error or the number of breakpoints. Note that the setting of our study falls under fixing the error.

    As the foundation of our research, we first introduce the important work by \cite{rossi2014piecewise}, which focuses on minimizing error in piecewise linear approximation for first-order loss functions with a fixed number of breakpoints.
    Their method divides the domain into \( n \) intervals and uses a new discrete random variable to calculate the conditional expectation for each interval.
    Based on this, various heuristics have been proposed for different distributions \citep{rossi2014piecewise-conf,rossi2015piecewise}.
    This approach is widely used, especially in the field of inventory management \citep{tunc2018extended,kilic2019heuristics,gutierrez2023stochastic,xiang2023mathematical}. While the method in \cite{rossi2014piecewise} focuses on guarantees for minimizing error, it does not fully explore the relationship between error and the number of breakpoints and is specifically tailored for normal distributions.
    In contrast, our research can determine the relationship between error and breakpoints a priori. Our study is applicable to both continuous and discrete distributions and can also be used for scenario reduction in scenario data.

    There is a history of research on minimizing error in piecewise linear approximation for convex functions with fixed breakpoints \citep{cox1971algorithm, gavrilovic1975optimal, imamoto2008recursive, liu2021optimal}, which serve as the basis for \cite{rossi2014piecewise}.
    Similar to \cite{rossi2014piecewise}, most of these works focus solely on minimizing error without delving into the relationship between error and the number of breakpoints. The sole exception is the study by \cite{liu2021optimal}, which provides the first trade-off between error and breakpoints.
    Although their study achieves an error analysis and trade-off similar to ours, it is difficult to apply to the loss function for the following two reasons:
    1.Their trade-off analysis relies on derivative information, making it inapplicable to general loss functions composed of discrete distributions.
    2.Their algorithm requires simplicity in the derivative form of the target function for error computation, making it challenging to directly apply to general loss functions.
     Additionally, most of the studies mentioned here require solving non-convex optimization problems, leaving the computational complexity unknown.
    In contrast, our research has the advantage of bounding the number of iterations of our algorithms by the number of breakpoints.

    Next, we review studies that aim to minimize the number of breakpoints given a specified allowable error, which, as stated in \cite{ngueveu2019piecewise}, are much fewer in number compared to the other setting. \cite{ngueveu2019piecewise} proposed a method for minimizing breakpoints that can be applied to a specific MILP. While their results are substantially different from ours, their motivation to focus on the trade-off between error and the number of breakpoints is similar to ours. Here, we discuss the study by \cite{rebennack2015continuous}, which is most closely related to our research.
    In \cite{rebennack2015continuous}, they formulated a MILP optimization problem to minimize the number of breakpoints for general univariate functions, given an allowable error \( \epsilon \).
    Additionally, they proposed a heuristic for determining an adjacent breakpoint such that the error within an interval does not exceed \( \epsilon \) when a specific breakpoint is given.
    Even though the heuristic algorithm's approach is similar to ours, it does not provide bounds on the number of breakpoints, making our results non-trivial.

    Although it is slightly outside the scope, we discuss the research on the scenario generation approach, which is another popular method for dealing with a random variable \( X \) in stochastic optimization. In the scenario generation approach, a simpler discrete random variable \( Y \) is generated to approximate a random variable \( X \), with each value of \( Y \) referred to as a scenario. Sample average approximation, such as Monte Carlo Sampling, is one of the most widely used methods in scenario generation \citep{shapiro2003monte}. Recall that our proposed method approximates \( X \) by a new discrete distribution \( \tilde{X} \) using conditional expectations based on \cite{rossi2014piecewise}. Thus, it can also be viewed as a scenario generation approach that guarantees the number of generated scenarios and the error. Note that, in existing studies in scenario generation, the function incorporating \( X \) (in our case, \( \ell_X \)) is not specified, and the error is evaluated using some form of distance, such as the Wasserstein distance, between the target random variable \( X \) and the generated variable \( Y \). This is significantly different from our study, which evaluates the absolute error between \( \ell_X \) and \( \ell_Y \) when the function is specified. For the main methods related to scenario generation, please refer to \cite{lohndorf2016empirical}.

    An overview and structure of this study can be found in the next section; please refer to it for details.

    \newpage

    \section{Research Setting and Outline}
    In this section, we outline the research, discuss its structure, and introduce the notation used throughout the study.

    \subsection{Scope and Setting}
    The following box summarizes the settings of this study.

    \begin{mdframed}
        \begin{description}
            \item [Input:]\leavevmode
            \begin{enumerate}
                \item a half open interval $(a, b] \subseteq \R$
                \item a one-dimensional random variable $X$
                \item a univariate function $f_X: (a, b] \rightarrow \R; \ f_X(s) = \E[ \min (s, X)]$
            \end{enumerate}
            \item [Task:]\leavevmode
            \begin{enumerate}
                \item Create a new discrete random variable $\tilde{X}$ from $X$ such that $f_{\tilde{X}}$ is a piece-wise linear approximation of $f_X$, which is based on \cite{rossi2014piecewise}.
                \item Analyze \( f_{\tilde{X}} \) based on the following evaluation criteria.
            \end{enumerate}
            \item [Evaluation:]\leavevmode
            \begin{enumerate}
                \item the absolute approximation error defined by
                $$
                e_{X, \tilde{X}} \coloneqq \max_{s \in (a,b]} |f_{\tilde{X}}(s) - f_X(s)|
                $$
                \item the number of the breakpoints of $\tilde{f}$
            \end{enumerate}
        \end{description}
    \end{mdframed}
    For the setting, the following two points need further clarification:
    \begin{enumerate}
        \item Target Function:\\
        We can show that a general loss function $\ell_X$ can be expressed as the sum of an affine function and $\E[\min(s, X)]$ (please refer to \ref{sec_reduction}). Without loss of generality, this reduction allows us to focus on the piecewise linear approximation of a simple function $f_X(s) := \E[\min(s, X)]$. Thus, in the rest of paper, we focus on $f_X$. We use the half-open interval \((a, b]\) as the approximation domain in order to simplify the discussion when dealing with discrete distributions.

        \item Method of Piecewise Linear Approximation: \\
        In this study, we employ the method proposed by \cite{rossi2014piecewise} for piecewise linear approximation. In their method, we transform the target random variable \(X\) into a discrete random variable \(\tilde{X}\) that takes on fewer possible values than \(X\) as follows. First, we divide $\R$ into $n$ consecutive regions $\mathcal{I} = (I_1, \dots, I_n)$. We construct $\tilde{X}$ from $\mathcal{I}$, which takes the conditional expectation $ \E[X \mid X \in I]$ with probability $P(X \in I)$ for each $I \in \mathcal{I}$.
        Then we consider the loss function \(f_{\tilde{X}}\) obtained by replacing \(X\) with \(\tilde{X}\) in \(f_X\), which can be shown to be a piecewise linear function with $n+1$ breakpoints. It implies that the piecewise linear function $f_{\tilde{X}}$ is uniquely determined once a partition $\mathcal{I}$ is fixed.
    \end{enumerate}

    \subsection{Structure of the Paper}
    The structure of the rest of this paper is as follows: In Section 3, we analyze the error \( e_{X, \tilde{X}} \) under the given partition \( \mathcal{I} \). In Section 4, we propose algorithms for generating a partition \( \mathcal{I} \) with an error less than \( \epsilon \) and derive upper bounds on the number of breakpoints for the induced piecewise linear functions. In Section 5, we derive lower bounds on the number of breakpoints based on the framework of our analysis. In Section 6, we compare the actual errors and the number of breakpoints with their theoretical values through numerical experiments and discuss the results. In Section 7, we give our conclusion.

    \subsection{Notation and Assumptions}
    We assume that all random variables treated in this paper are real-valued random variables with expected values. Random variables can be either continuous or discrete unless explicitly stated. For a random variable $X$, its distribution function is denoted by $p_X: \R \rightarrow [0,1]$. For a half-open interval $I=(a, b] \subseteq \mathbb{R}$, we define the part of expectation of $\E[X]$ as
    $$
    \E_{a,b}[X]  :=  \left \{\begin{array}{ll}
                                 \int_a^b \ p_X(x) x \, dx          & \text{if $X$ is continuous}, \\
                                 \sum_{x \in (a, b] \cap S} P(X=x)x & \text{if $X$ is discrete},   \\
    \end{array}
    \right .
    $$
    where $S$ is the support of $X$.

    \section{Analysis of Approximation Error}

    In this section, we introduce the piecewise linear method used in the study and evaluate the error in its approximation.

    \subsection{Piecewise Linear Approximation Framework}
    First, we introduce a piecewise linear approximation framework for the function $f_X(s) = \E[\min(s, X)]$ on some half interval $(a, b] \subseteq \mathbb{R}$, which is based on \cite{rossi2014piecewise}.

    Let $\mathcal{I}=(I_0, I_1, \dots I_n, I_{n+1})$ be a partition of $\R$ such that
    \begin{itemize}
        \item $I_j = (a_j, b_j]$ for $j \in \{0, 1, \dots, n, n+1\}$ such that $b_j = a_{j+1}$ for $j \in \{0, 1, \dots, n\}$
        \item $a_0 = - \infty$, $b_0 = a$, $a_{n+1}=b$ and $b_{n+1} = \infty$.
    \end{itemize}
    To simplify the discussion, we assume that for all $I \in \mathcal{I}$, $P(X \in I)$ is positive.
    We consider a new discrete random variable $\tilde{X}$ with $\mathcal{I}$ to approximate $X$ as follows.

    \begin{definition}
        \label{def_generation}
        A discrete random variable $\tilde{X}$ is said to be induced by a random variable $X$ with $\mathcal{I}$ if $\tilde{X}$ takes $ \E[X \mid X \in I_j ]$ with probability $P(X \in I_j)$  for $j \in \{0, 1, \dots, n, n+1\}$.
    \end{definition}

    With $\tilde{X}$ instead of $X$ of $f_X$, $f_{\tilde{X}}$ is written as
    \begin{equation}
        \label{def_f_tilde}
        f_{\tilde{X}}(s) \coloneqq \E [ \min(s, \tilde{X}) ] = \sum_{j=0}^{n+1} (P(X \in I_j) \cdot \min (s, \mu_j)),
    \end{equation}
    where for $j \in \{0, 1, \dots, n, n+1\}$ we define
    \begin{equation}
        \mu_j \coloneqq \E[X \mid X \in I_j].
    \end{equation}
    We easily show that $f_{\tilde{X}}$ is a continuous piecewise linear function as follows.
    \begin{proposition}
        $f_{\tilde{X}}$ on $(a, b]$ is a continuous piecewise linear function on with $n$ breakpoints.
    \end{proposition}
    \begin{proof}
        Assume that $\mu_i < s \leq \mu_{i + 1} $ for some $i \in \{0, 1, \dots, n\}$. Then we have
        $$
        f_{\tilde{X}}(s) = \sum_{j = 0}^i (P(X \in I_j) \cdot \mu_j) + s \cdot \sum_{j = i+1}^{n+1} P(X \in I_j).
        $$
        Therefore, $f_{\tilde{X}}$ is a continuous piecewise linear function, whose breakpoints are $\mu_1, \dots, \mu_n$.
    \end{proof}
    Based on the above results, the piecewise linear function $f_{\tilde{X}}$ is uniquely determined once a partition $\mathcal{I}$ of $(a, b]$ and the associated random variable $\tilde{X}$ have been specified.

    \subsection{Analysis of Approximation Error}

    Now, we evaluate the absolute error between the piecewise linear function $f_{\tilde{X}}$ and $f_X$ defined as
    \begin{equation}
        \label{def:approximation_error}
        e_{X, \tilde{X}} = \max_{s \in (a, b]} |f_{\tilde{X}}(s) - f_X(s)|.
    \end{equation}
    Since $(I_1, \cdots, I_n)$ is a partition of $(a, b]$, by taking the maximum of the errors in each region, $e_{X, \tilde{X}}$ is rewritten as
    \begin{align*}
        e_{X, \tilde{X}} &= \max_{j \in \{1,\dots, n\}} \max_{s \in I_j}|   f_{\tilde{X}}(s) - f_X(s)|\\
        &=  \max_{j \in \{1,\dots, n\}} \Delta_X(I_j),
    \end{align*}
    where we define
    \begin{equation}
        \label{def_delta}
        \Delta_X(I_j) \coloneqq \max_{s \in I_j}|   f_{\tilde{X}}(s) - f_X(s)| \quad (j \in \{1, \dots, n\}).
    \end{equation}
    \begin{remark}
        Due to strict concavity of $f_X$, \( e_{X, \tilde{X}} \) is minimized when all \( \Delta_X(I_j) \) have the same value \citep{imamoto2008recursive}.
        Utilizing this property, research led by \cite{rossi2014piecewise} and others employ an approach that solves a nonlinear optimization problem to make all \( \Delta_X(I_j) \) equal. These analyses are only valid for the optimal partition.
        In our study, we evaluate the approximation error for any partition for $(a, b]$, which makes it possible to evaluate the piecewise linear approximation function that provides a good but non-optimal approximation.
    \end{remark}

    From here on, we will evaluate \( \Delta_X(I_k) \) for fixed $k \in \{1, \dots n\}$. For $s \in (a, b]$, we transform \( f_X \) as follows.
    \begin{align*}
        f_X(s) &= \E[\min(s, X)] = \sum_{j=0}^{n+1}\E_{(a_j, b_j]}[\min (s, X)]\\
        &= \sum_{j=0}^{n+1}P(X \in I_j) \E [ \min (s, X_{I_j})],
    \end{align*}
    where we define the conditional random variable when $X \in I_j$ as
    \begin{equation}
        X_{I_j} \coloneqq X \mid X \in I_j \quad (j \in \{1, \dots, n\}).
    \end{equation}
    When $t \in I_k$,  for any $j \neq k$, we have
    $$
    \E [ \min(t, X_{I_j})] = \min(t, \mu_j)
    $$
    because we see
    $$
    \E [ \min(t, X_{I_j})] =
    \begin{cases}
        \mu_j \quad  (\leq b_j \leq b_k \leq t) & \text{if} \ j < k, \\
        t \quad (\leq b_k \leq a_j \leq  \mu_j) & \text{if} \ j > k.
    \end{cases}
    $$
    Comparing \( f_{\tilde{X}}(s) \) as defined in \eqref{def_f_tilde} with \( f_X \),
    each term corresponding to \( j \neq k \) has the same value in both \( f_X \) and \( f_{\tilde{X}} \). Thus, by subtracting \( f_X(t) \) from \( f_{\tilde{X}}(t) \), we obtain the following.
    \begin{equation*}
        f_{\tilde{X}}(t) - f_X(t) =  P(X \in I_k) (\min (t, \mu_k) - \min(t, \E[ \min (t, X_{I_k})] )) \geq 0,
    \end{equation*}
    where the last inequality is from Jensen's inequality since $f_{X_{I_k}}$ is proven to be concave in Lemma \ref{lemma:concave}. By substituting the above equation into the definition of \( \Delta_{X} \) given by \eqref{def_delta}, we obtain the following lemma.

    \begin{lemma}
        \label{lemma:delta}
        For each $j \in \{1, \cdots, n\}$,
        \begin{equation}
            \label{eq:delta}
            \Delta_X (I_j) = P(X \in I_j) \cdot \max_{s \in I_j} (   \min(s, \mu_j) -\E [ \min(s, X_{I_j})]    ).
        \end{equation}
    \end{lemma}
    We can show that the maximum value in $\Delta_X(I_j)$ is attained at $s = \E[X \mid X \in I_j]$ and its value is calculated as follows, which is proved in the next subsection. The relationship is illustrated in Figure~\ref{fig_1}.

    \begin{figure}[ht]
        \centering 
        \includegraphics[width=0.7\textwidth]{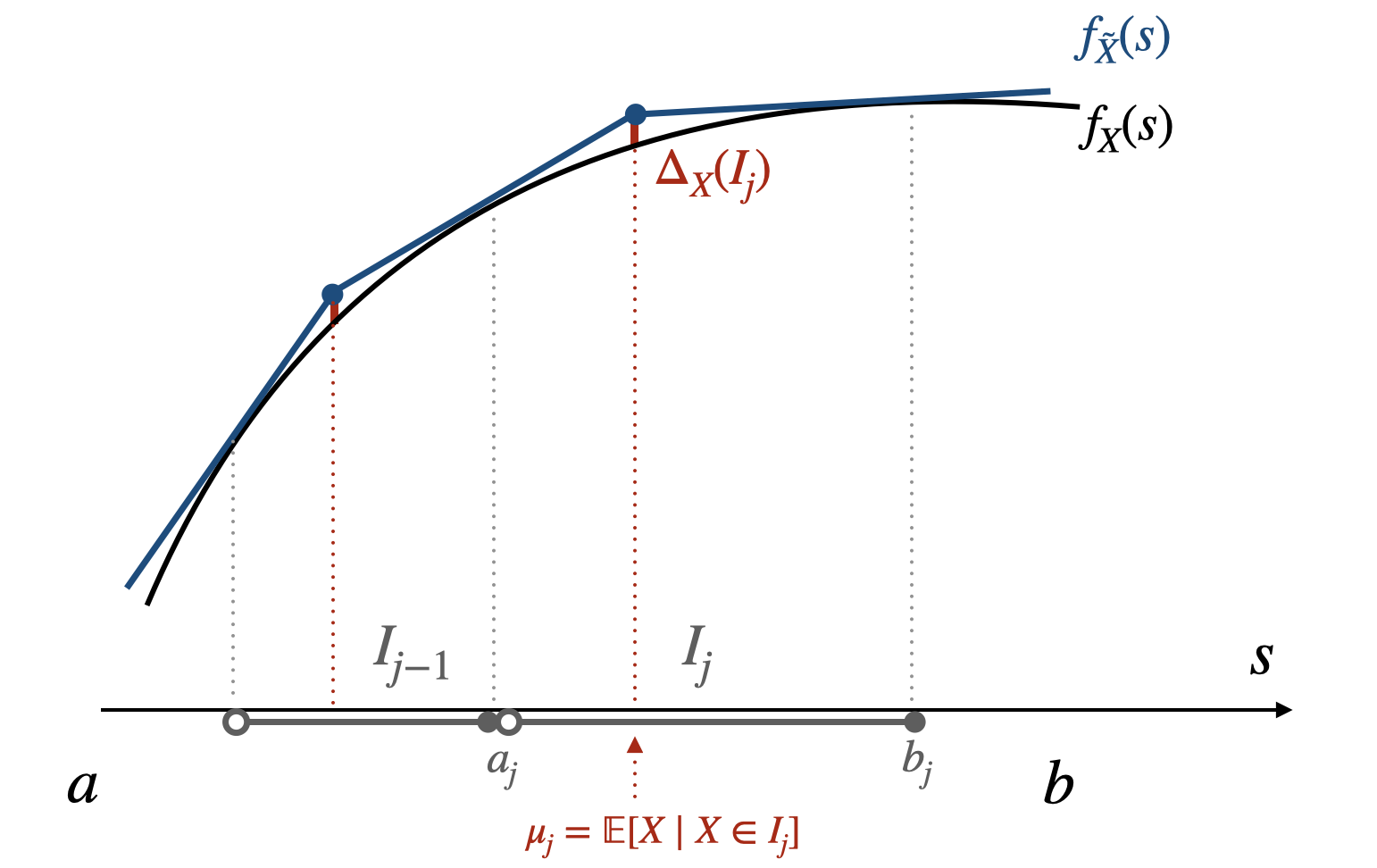} 
        \caption{Relationship of $f_X$ and $f_{\tilde{X}}$} 
        \label{fig_1} 
    \end{figure}

    \begin{lemma}
        \label{lemma:one-break}
        For an interval $I_j = (a_j, b_j]$,
        $$
        \max_{s \in I_j}  (   \min(s, \mu_j) -\E [ \min(s, X_{I_j})] )  = \E_{a,\mu_j}[\mu_j -X_{I_j}] \leq \frac{b_j-a_j}{4}.
        $$
    \end{lemma}
    From the above lemma, we have the following theorem, where the error in each interval can be analytically computed and is bounded by a value proportional to the product of the length of each region and the probability within that region.

    \begin{theorem}
        \label{theorem:delta}

        For a interval $I_j = (a_j, b_j] \ (j \in \{1, \dots,n\})$,        $$
        \Delta_X(I_j)  = \E_{a_j,\mu_j}[\mu_j -X] \leq P(X \in I_j) \cdot \frac{b_j -a_j}{4}.
        $$
    \end{theorem}
    \begin{proof}

        From the inequality in Lemma \ref{lemma:one-break} and Lemma \ref{lemma:delta}, we have
        $$
        \Delta_{X}(I_j) \leq P(X\in I_j) \cdot \frac{b_j -a_j}{4}.
        $$
        Next, we show the equality. From the equality in Lemma \ref{lemma:one-break} and Lemma \ref{lemma:delta}, we have
        $$
        \Delta_{X}(I_j) = P(X \in I_j) \cdot \E_{a_j, \mu_j}[\mu_j - X_{I_j}].
        $$
        Since
        $$
        \E_{a_j, \mu_j}[\mu_j - X_{I_j}] = \frac{\E_{a_j, \mu_j}[\mu_j - X]}{P(X \in I_j)},
        $$
        we obtain
        $$
        \Delta_X(I_j) =  \E_{a_j, \mu_j}[\mu - X].
        $$

    \end{proof}

    We finally derive the following result.
    \begin{corollary}
        \label{theorem:approx}
        \begin{equation}
            e_{X, \tilde{X}} =  \max_{j \in \{1, \dots n\} } \E_{a_j,\mu_j}[\mu_j - X]\leq \max_{j \in \{1, \dots n\}} \left(P(X \in I_j) \cdot \frac{b_j - a_j}{4}\right).
        \end{equation}
    \end{corollary}

    \subsection{Proof of Lemma \ref{lemma:one-break}: Approximation Error of Piecewise Linear Approximation with One Breakpoint}

    In this subsection, we provide a proof of Lemma \ref{lemma:one-break}. Let $Y$ be a real-valued random variable whose support is a subset of $(a, b] \subseteq \mathbb{R}$. Let $\mu$ be the mean value of $Y$, that is, $\mu := \E[Y]$. Consider the piecewise linear approximation function to $f_Y(s) = \E[\min(s, Y)]$, whose breakpoint is only the mean value $\E[Y]$. Then, this function is $\min(s, \E[Y])$ and we define its approximation error at a point $s \in (a, b]$ as
    \begin{equation}
        \delta_Y (s) \coloneqq |\min(s, \E[Y]) - f_Y(s)|.
    \end{equation}
    To show the claim of Lemma \ref{lemma:one-break}, it suffices to show that the maximum value of $\delta_Y(s)$ is equal to $\E_{a,\mu}[\mu -Y]
    $ and it is bounded by $\frac{b-a}{4}$. The following lemma shows that $\delta_Y(s)$ attains its maximum value at $s=\mu$ and its maximum value is  $\E_{a,\mu}[\mu -Y]$. The relationship is illustrated in Figure~\ref{fig_2}.

    \begin{figure}[ht]
        \centering 
        \includegraphics[width=0.7\textwidth]{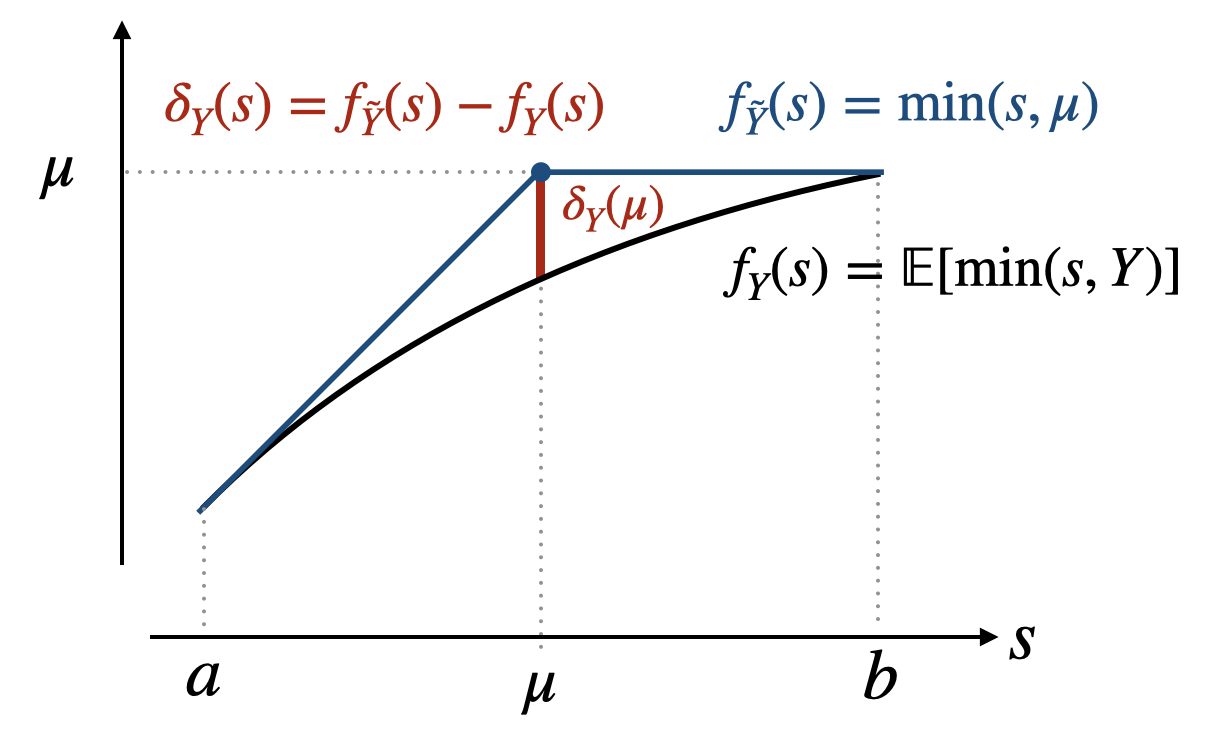} 
        \caption{$f_Y$ and $f_{\tilde{Y}}$} 
        \label{fig_2} 
    \end{figure}

    \begin{lemma}
        \label{lemma:delta_max}
        \begin{equation}
            \max_{s \in (a,b]} \delta_Y(s) = \delta_Y(\mu) = \E_{a,\mu}[\mu -Y]
        \end{equation}
    \end{lemma}
    \begin{proof}
        Assume $s \in (a, b]$. Then, $f_Y$ is written as
        $$
        f_Y(s) = \E_{a,s}[Y] + \E_{s,b}[s].
        $$
        Thus, if $s \in (a, \mu]$, from $\min(s, \mu)=s$, we see that
        \begin{align*}
            \delta_Y(s) &= s - (\E_{a,s}[Y] + \E_{s,b}[s])\\
            &=\E_{a,s}[s] + \E_{s,b}[s]  - \E_{a,s}[Y] - \E_{s,b}[s]\\
            &= \E_{a,s}[s-Y].
        \end{align*}
        Similarly, if $s \in (\mu, b]$, from $\min(s, \mu) = \mu$, we see that
        \begin{align*}
            \delta_Y(s) &= \mu - (\E_{a,s}[Y] + \E_{s,b}[s])\\
            &=\E_{a,s}[Y] + \E_{s,b}[Y]  - \E_{a,s}[Y] - \E_{s,b}[s]\\
            &= \E_{s,b}[Y-s].
        \end{align*}
        Hence, we obtain
        $$
        \delta_Y(s) =  \left \{\begin{array}{ll}
                                   \E_{a,s}[s-Y] & \text{if} \  s \in (a, \mu], \\
                                   \E_{s,b}[Y-s] & \text{if} \  s \in (\mu, b]. \\
        \end{array}
        \right .
        $$
        For $s \in (a, \mu)$ and $s +\epsilon \in  (a, \mu]$ such that $\epsilon >0$, we see that
        \begin{align*}
            \E_{a, s + \epsilon}[s + \epsilon - Y] &= \E_{a, s}[s + \epsilon - Y] + \E_{s, s+\epsilon}[s + \epsilon - Y]\\
            &= \E_{a, s}[s - Y] + \E_{a, s}[\epsilon] + \E_{s, s+\epsilon}[s + \epsilon - Y]\\
            &\geq \E_{a, s}[s - Y].
        \end{align*}
        Thus, $\delta_Y(s)$ is increasing for $s \in (a, \mu]$, and decreasing for $s \in (\mu, b]$ as well. Also, it is continuous since $f_{Y}$ and $f_{\tilde{Y}}$ are continuous from Lemma \ref{lemma:concave}. Therefore, it attains its maximum value at $s=\mu$.
    \end{proof}
    We finally obtain the upper bound as follows.
    \begin{theorem}
        \label{theorem:one-point}
        \begin{equation}
            \max_{s \in (a, b]} \delta_Y(s)  \leq \frac{b-a}{4}.
        \end{equation}
    \end{theorem}

    \begin{proof}
        Since $\max_{s \in (a, b]}\delta_Y(s) = \delta_Y(\mu)$ from Lemma \ref{lemma:delta_max}, we have
        \begin{align}
            \delta_Y(\mu) &= \E_{a, \mu}[\mu - Y] \notag \\
            &= P(a < Y \leq \mu) \E[\mu - Y | a < Y \leq \mu] \notag \\
            &= P_A (\mu -\mu_A) \label{eq:last},
        \end{align}
        where $P_A = P(Y \leq \mu)$, $\mu_A = \E[Y \mid Y \leq \mu]$.
        Define $P_B = P( \mu <  Y \leq  b)$ and $\mu_B = \E [Y \mid \mu < Y \leq b]$, where we assume that $P_B > 0$ since $\delta_Y(\mu) = 0$ holds when $P_B = 0$. Then, we have the following relation:
        \begin{align*}
            P_A + P_B &= 1, \\
            P_A \mu_A + P_B \mu_B  &= \mu.
        \end{align*}
        Thus, we have $P_A = \frac{\mu_B - \mu}{\mu_B - \mu_A}$. Therefore, substituting $P_A$ into \eqref{eq:last}, we obtain the following equation:

        $$
        \delta_Y(\mu) = \frac{(\mu - \mu_A)(\mu_B - \mu)}{\mu_B - \mu_A}.
        $$
        Moreover, we obtain
        $$
        \frac{(\mu - \mu_A)(\mu_B - \mu)}{\mu_B - \mu_A} \leq \max_{\mu' \in [\mu_A, \mu_B]} \frac{(\mu' - \mu_A)(\mu_B - \mu')}{\mu_B - \mu_A} = \frac{\mu_B-\mu_A}{4},
        $$
        where the last equation is from that the maximum value is attained when $\mu' = \frac{\mu_B - \mu_A}{2}$.
        Since $\mu_B - \mu_A \leq b-a$ holds, we finally get $\delta_Y(\mu) \leq \frac{b-a}{4}$.
    \end{proof}
    Under certain assumptions, we can approximate $\delta_Y(\mu) \approx \frac{b-a}{8}$, which is useful in the design of algorithms in the next section.
    \begin{remark}
        \label{remark:uniform}

        Let $Z$ be a random variable having the probability density function $f: [a, b] \rightarrow \R$ and assume $f$ is absolutely continuous on $[a,b]$ and its first derivative $f'$ belongs to the Lebesgue space $L_\infty [a,b]$. From Theorem 2 in \cite{barnett2000some}, we have
        \begin{align*}
            \left|\E[Z] - \frac{a + b}{2}\right|  \leq \frac{(b-a)^3}{12} \| f'\|_{\infty}.
        \end{align*}
        Thus, when the first derivative of the probability function in $(a, b]$ is small enough, in the proof above, we can approximate the error as
        $$
        \delta_Y(\mu) = \frac{(\mu - \mu_A)(\mu_B - \mu)}{\mu_B - \mu_A} \approx \frac{b-a}{8},
        $$
        where
        $$
        \mu \approx \frac{a+b}{2}, \ \mu_A \approx \frac{a + \mu}{2} \ \text{and}\ \mu_B \approx \frac{\mu+b}{2}.
        $$
        Under this assumption, we can also approximate \( \Delta_{X}(I_j) \) as follows:
        \begin{equation}
            \label{eq:approx}
            \Delta_X(I_j) \approx P(X \in I_j) \cdot \frac{b_j - a_j}{8}.
        \end{equation}
    \end{remark}

    \newpage

    \section{Partition Algorithms}

    In this section, we propose a partition algorithm that guarantees a bounded number of breakpoints while keeping the error below \( \epsilon \). Our algorithm is shown in Algorithm \ref{partition-algorithm} and based on the results of Section 3.

    \begin{algorithm}[H]
        \caption{Partition Algorithm}
        \label{partition-algorithm}
        \begin{algorithmic}[1]
            \REQUIRE  \ \\
            - error bound function, $B \in \{B_{\text{exact}}, B_{1/4}, B_{1/8} \}$ defined by \eqref{def_B} \\
            - support of the random variable, S $\subseteq \R$ \\
            - target interval $(a, b] \subseteq \mathbb{R}$\\
            - acceptable error $\epsilon > 0$\\
            \ENSURE a partition of $\R$, $I=(I_0, I_1, \dots, I_n, I_{n+1})$, such that $(I_1, \dots, I_n)$ is also a partition of $(a,b]$
            \STATE $a_1 \leftarrow a$
            \STATE $j \leftarrow 1$
            \WHILE{$B(a_j, b)> \epsilon$}
            \STATE $\displaystyle b_j \leftarrow \max_{y \in (a_j, b) \cap S} y \ \text{s.t.} \ B(a_j, y) \leq \epsilon$ \label{find_b}
            \STATE $a_{j+1} \leftarrow b_j$
            \STATE $I_j \leftarrow (a_j, b_j]$
            \STATE $j \leftarrow j+1$
            \ENDWHILE
            \STATE $b_j \leftarrow b$
            \STATE $I_j \leftarrow (a_j, b_j]$
            \STATE $n \leftarrow j$
            \STATE $I_0 \leftarrow (-\infty, a]$
            \STATE $I_{n+1} \leftarrow (b, -\infty)$
            \RETURN $\mathcal{I}=(I_0, I_1, \dots, I_n, I_{n+1})$
        \end{algorithmic}
    \end{algorithm}

    The input to the Partition Algorithm consists of a tolerance \(\epsilon > 0\), an interval \((a, b]\), and a bound function \( B(x, y) \) that roughly represents the error in the piecewise linear approximation over the interval \( (x, y] \). $B$ is chosen in \( \{B_{\text{exact}}, B_{1/4}, B_{1/8}\} \) defined as
    \begin{align}
        \begin{split}
            B_{\text{exact}}(x, y) & \coloneqq \E_{x, \mu}[\mu - X] \  \ (\mu = \E[X \mid X \in (x, y]])   \\
            B_{1/4}(x, y) &\coloneqq   P(X \in (x, y]) \cdot \frac{y-x}{4}\\
            B_{1/8}(x, y) &\coloneqq  P(X \in (x, y]) \cdot \frac{y-x}{8},
        \end{split}
        \label{def_B}
    \end{align}
    where $B_{\text{exact}}$ and $ B_{1/4}$ are derived from the equality and inequality of Theorem \ref{theorem:delta}, respectively and $B_{1/8}$ is derived from Remark \ref{remark:uniform}.
    For simplicity, we refer to the Partition algorithm using the bound function \( B \) as Algorithm \( B \). The output is a partition of $\R$.

    The idea behind the algorithm is quite simple. We start with \(a_1 = a\), and in the \(j\)-th iteration, with \(a_j\) already determined, we choose \(b_j\) to be the largest value such that an error function \(B(a_j, b_j)\) does not exceed \(\epsilon\). In this study, we assume that we can find \( b_j \) exactly. The validity of this assumption and an actual method for finding \( b_j \) will be discussed later. Our algorithms bear a resemblance to the heuristic algorithms presented by \cite{rebennack2015continuous}, in which the next maximum breakpoint is selected to ensure that the error does not exceed $\epsilon$. The key distinction is that, in our algorithms, we determine the breakpoints indirectly by defining the intervals.

    From Theorem \ref{theorem:delta}, we see the following relation
    $$
    \Delta_X((x,y]) = B_{\text{exact}}(x, y) \leq 2B_{1/8}(x, y) = B_{1/4}(x, y).
    $$
    Therefore, we have the following result about approximation errors from Colorollary \ref{theorem:approx}.
    \begin{theorem}

        Let $\tilde{X}$ be the discrete random variable induced by piecewise linear approximation with the output $\mathcal{I}$ of Algorithm \ref{partition-algorithm}.
        Depending on the setting of \( B \), the following results are obtained:
        \begin{alignat*}{2}
            e_{X, \tilde{X}} &\leq \epsilon  &\quad \text{ if } \quad & B = B_{\text{{exact}}} \text{ or } B=B_{1/4}, \\
            e_{X, \tilde{X}} &\leq 2\epsilon &\quad \text{ if } \quad & B=B_{1/8}.
        \end{alignat*}
    \end{theorem}

    The remaining issue is how large the upper bound of the number of breakpoints for each algorithm is.
    To summarize, the results obtained in this study can be compiled in the following table (Table \ref{tab:breakpoints}).
    \begin{table}[htb]
        \centering
        \begin{tabular}{cccc}
            \toprule
            Algorithm & error & \multicolumn{2}{c}{breakpoints} \\
            &                 & continuous                                           & discrete                                             \\
            \midrule
            \( B_{\text{exact}} \) and  \( B_{1/4} \) & \( \epsilon \)  & \( \frac{1}{2} \sqrt{\frac{b-a}{\epsilon}}\)         & $\sqrt{\frac{b-a}{\epsilon}}$                         \\
            \( B_{1/8} \)                           & \( 2\epsilon \) & \( \frac{1}{2\sqrt{2}}\sqrt{\frac{b-a}{\epsilon}} \) & \( \frac{1}{\sqrt{2}} \sqrt{\frac{b-a}{\epsilon}} \)  \\
            \bottomrule
        \end{tabular}
        \caption{Summary of Properties for Each Algorithm}
        \label{tab:breakpoints}
    \end{table}
    As can be seen from Table \ref{tab:breakpoints}, the bounds on the number of breakpoints differ depending on whether the target random variable is continuous or discrete. Furthermore, \( B_{\text{exact}} \) and \( B_{1/4} \) yield equivalent results in terms of theoretical guarantees.

    Apart from the above differences, each algorithm has its own unique characteristics. We can show that Algorithm \( B_{\text{exact}} \) is guaranteed to have the minimum number of breakpoints. However, for continuous random variables, the calculation of \( B_{\text{exact}} \) involves numerical integration for conditional expectations, making its implementation costly and potentially introducing numerical errors. Algorithm \( B_{1/8} \) theoretically could have an error up to \( 2\epsilon \), but in practice, it behaves almost identically to \( B_{\text{exact}} \) (see the section on numerical experiments for details).

    Finally, we outline the remaining structure of this section. In Section 4.1, we discuss the specific implementation methods for finding \( b_j \). In Section 4.2, we prove the optimality of Algorithm \( B_{\text{exact}} \). Lastly, in Section 4.3, we provide bounds on the number of breakpoints for each algorithm.

    \subsection{Implementation of Finding Next Point}
    We discuss the actual implementation methods for finding \( b_j \) even though we assume that \( b_j \) is determined exactly in line \ref{find_b} of Algorithm 1,

    Unless otherwise specified, \( B \) can be any of \eqref{def_B}, where we assume $B(x, x) = 0$ for any $x$.
    Here, we describe the implementation method for finding \( b_j \) as performed in line \ref{find_b} of the algorithm. First, for \( x \in (a, b) \), we define the following function \( L_x \colon (x, b] \rightarrow \mathbb{R} \):
    \begin{equation}
        L_x(y) = B(x, y) - \epsilon.
    \end{equation}
    At the time of the \( j \)-th iteration in line \ref{find_b}, \( L_{a_j}(a_j) = -\epsilon < 0 \) and \( L_{a_j}(b) > 0 \) hold.

    We first consider the case where \( X \) is continuous. For simplicity, we assume that the support of \( X \) is \( \mathbb{R} \). We assume that \( L_{a_j} \) is continuous and strictly monotonically increasing. This is an assumption that holds for standard distributions. At this time, there exists only one \( y \) such that \( L_{a_j}(y) = 0 \), that is, \( B(a_j, y) = \epsilon \), within \( (a_j, b) \). This can be realized by calling \( B \) \( O(\log (\frac{b-a}{d})) \) times using binary search for a small allowable error \( d > 0 \).

    Next, we consider the case where \( X \) is discrete. Let \( Z = \{x_1, \dots, x_K\} \) where \( x_1 < \dots < x_K \) be the intersection of the support \( S \) of \( X \) and \( (a, b] \). We assume that $a_j = x_{k[j]} \in Z$. We can find the largest $y$ from $\{x_{k[j]+1}, \dots, x_{K-1}\}$ satisfying \( L_{a_j}(y) \leq \epsilon \) by \( O(\log(b-a)) \) calls to \( B \). Finally, we note that special handling is needed for $B_{1/4}$ and $B_{1/8}$ when $\epsilon$ is extremely small.
    In such cases, it is possible that no $y$ satisfying $B(a_j, b_j) \leq \epsilon$ exists within $\{x_{k[j]+1}, \dots, x_{K-1}\}$.
    In this situation, we set $b_j = x_{k[j] + 1}$.
    Although this would make $B(a_j, b_j) > \epsilon$, the actual error on $(a_j, b_j]$ can still be bounded below $\epsilon$ as the exact error $B_{\text{exact}}(a_j,b_j)=0$.

    \subsection{Optimality}

    In this section, we show the optimality of Algorithm \( B_{\text{exact}} \). Specifically, we show that the output of this algorithm is an optimal solution to the following optimization problem as follows: for given $\epsilon > 0$ and $(a, b] \subseteq  \R$:
    \begin{align}
        \label{opt_breakpoint}
        \begin{array}{lll}
            \text{minimize}   & n                                                                \\
            \text{subject to} & \Delta_X(I_j) \leq \epsilon \quad \forall i \in \{1, \dots, n\}, \\
        \end{array}
    \end{align}
    where $n \in \mathbb{N}$ and a partition $\mathcal{I}=(I_0, \dots, I_{n+1})$ are decision variables.
    \begin{theorem}
        \label{theorem:opt}
        Let $\mathcal{I} = (I_0,I_1, \dots, I_n, I_{n+1})$ be the output of Algorithm $B_{\text{exact}}$. Then, $\mathcal{I}$ is an optimal solution of \eqref{opt_breakpoint}.
    \end{theorem}
    \begin{proof}
        Let $\mathcal{I} = (I_0,I_1, \dots, I_n, I_{n+1})$ be the output of Algorithm \ref{partition-algorithm}, where we use the notations $I_j = (a_j, b_j]$ for $j \in \{1, \dots, n\}$. We derive a contradiction by assuming that there is a feasible solution $\mathcal{I}^* = (I^*_0, \dots, I_{m+1}^*)$ such that $m < n$, where we also use the notation $I_j^* = (a_j^*, b_j^*)$. Without loss of generality, we assume that $a_j^*, b_j^*$ are chosen in the support $S$ of $X$.

        We prove for any $j \in \{1, \dots, m\}$, $b_j^* \leq b_j$ by induction. When $j=1$, we see that $a_1=a_1^* = a$, $\Delta(a, b_1) \leq \epsilon$ and $\Delta(a, b_1^*) \leq \epsilon$. Since $b_1$ is the maximum in $S \cap (a, b]$ such that $\Delta(a, b_1) \leq \epsilon$, we have $b_1^* \leq b_1 $. Now, for $j \in \{2, \dots, m\}$, we assume that $b_{j-1} \leq b_{j-1}^*$. First, we consider the case when $b_j^* \leq a_j $. In this case $b_j^* \leq a_j < b_j$ holds. Second, we assume that $a_j \leq b_j^*$. From the assumption of induction, we have
        $$
        a_j^* = b_{j-1}^*\leq b_{j-1} = a_j.
        $$
        Therefore, $a_j^* \leq a_j \leq b_j^*$ holds. From this relation, then it leads
        $$
        \Delta_X(a_j, b_j^*)    \leq \Delta_X(a_j^*, b_j^*) \leq \epsilon.
        $$
        where we use Lemma \ref{lemma:lower_delta}. When $b_j^* > b_j$, it contradicts that $b_j$ is maximal in $(a_j, b] \cap S$ such that $\Delta_X(a_j,b_j) \leq \epsilon$. Thus, we have $b_j^* \leq b_j$ for all $j \in \{1, \dots, m\}$.

        When $j=m$ in the result above, we have the following contradiction
        $$
        b = b_m^* \leq b_m < b_n = b.
        $$
        Hence, for any feasible solution $\mathcal{I}^* = (I^*_0, \dots, I_{m+1}^*)$, $m \geq n$ holds, which implies $\mathcal{I}$ is optimal.
    \end{proof}

    \subsection{Upper Bounds of Breakpoints}
    In this section, we provide bounds on the number of output intervals for each algorithm. Note that we give proof of the bounds only for  \( B_{1/4} \) and \( B_{1/8} \) since the number of the breakpoints with Algorithm $B_{\text{exact}}$  is less than or equal to that of Algorithm $B_{1/4}$ from Theorem \ref{theorem:opt}.

    First,  we derive bounds for the case when $X$ is continuous.
    \begin{theorem}
        Assume that $X$ is continuous. Let $\mathcal{I}=(I_0, I_1, \dots I_n, I_{n+1})$ be the outputs of Algorithm \ref{partition-algorithm}. Let $\tilde{X}$ be the discrete random variable induced by piecewise linear approximation with $\mathcal{I}$. Depending on the setting of \( B \), the following inequalities holds:
        \begin{alignat*}{2}
            n &\leq \frac{1+P}{4} \sqrt{\frac{b - a}{\epsilon}} + 1  \leq \frac{1}{2} \sqrt{\frac{b - a}{\epsilon}} + 1&\quad \text{ if } \quad & B \in \{B_{\text{{exact}}}, B_{1/4}\}, \\
            n &\leq \frac{1+P}{4\sqrt{2}} \sqrt{\frac{b - a}{\epsilon }} + 1 \leq \frac{1}{2\sqrt{2}} \sqrt{\frac{b - a}{\epsilon }} + 1&\quad \text{ if } \quad & B=B_{1/8},
        \end{alignat*}
        where $P = P(X \in (a, b])$.
    \end{theorem}

    \begin{proof}
        Let $\mathcal{I} = (I_0,I_1, \dots, I_n, I_{n+1})$ be the output of Algorithm \ref{partition-algorithm}, where we use the notations $I_j = (a_j, b_j]$ for $j \in \{1, \dots, n\}$. Let $B \in \{ B_{1/4}, B_{1/8}\}$ and
        $$
        M =
        \begin{cases}
            1 & \text{if} \quad B = B_{1/4}, \\
            2 & \text{if} \quad B=B_{1/8}.
        \end{cases}
        $$
        For $j \in \{1, \dots, n\}$, define $p_j = P(X \in I_j)$ and $r_j = \frac{b_j -a_j}{b-a}$. From $\sum_{j=1}^{n-1} p_j \leq P$ and $\sum_{j=1}^{n-1}r_j \leq 1$, we have
        $$
        \sum_{j=1}^{n-1} (p_j + r_j) \leq 1+P.
        $$
        Now, we will get an upper bound of $n$ by estimating a lower bound of $p_j + r_j$.
        For any $j \in \{1, \dots, n-1\}$, since $B(a_j, b_j) = \epsilon$ holds in Line \ref{find_b} in Partition Algorithm, we have
        $$
        \epsilon = P(X \in I_j) \frac{b_j-a_j}{4M} = \frac{p_jr_j(b-a)}{4M}.
        $$
        Thus, we have
        $$
        p_jr_j =\frac{4M\epsilon}{b-a}.
        $$
        From the relationship of the geometric mean, we obtain
        $$
        p_j + r_j \geq 2\sqrt{p_jr_j} = 4 \sqrt{\frac{M\epsilon}{b-a} }.
        $$
        By summing \( j \) from \( 1 \) to \( n-1 \), we obtain the following expression:

        $$
        4(n-1) \sqrt{\frac{M\epsilon}{b -a}} \leq \sum_{j=1}^{n-1}(p_j + r_j) \leq 1+P
        $$
        Therefore,
        $$
        n \leq   \frac{1+P}{4\sqrt{M}}\sqrt{\frac{b - a}{\epsilon}} + 1.
        $$
    \end{proof}

    Next, we derive bounds for the case when $X$ is discrete.
    \begin{theorem}\label{theorem_discrete}
        Assume that $X$ is discrete. Let $\mathcal{I}=(I_0, I_1, \dots I_n, I_{n+1})$ be the outputs of Algorithm \ref{partition-algorithm}. Let $\tilde{X}$ be the discrete random variable induced by piecewise linear approximation with $\mathcal{I}$. Depending on the setting of \( B \), the following inequalities holds:
        \begin{alignat*}{2}
            n &\leq \frac{1+P}{2} \sqrt{\frac{b - a}{\epsilon}} + 1  \leq  \sqrt{\frac{b - a}{\epsilon}} + 1&\quad \text{ if } \quad & B = B_{\text{{exact}}} \text{ or } B=B_{1/4}, \\
            n &\leq \frac{1+P}{2\sqrt{2}} \sqrt{\frac{b - a}{\epsilon }} + 1 \leq \frac{1}{\sqrt{2}} \sqrt{\frac{b - a}{\epsilon }} + 1&\quad \text{ if } \quad & B=B_{1/8},
        \end{alignat*}
        where $P = P(X \in (a, b])$.
    \end{theorem}

    \begin{proof}
        Recall that the endpoints of the interval $I_j = (a_j, b_j]$, $a_j$ and $b_j$, are chosen from within the support of $X$, denoted as $\{x_1, \dots, x_K\}$. In the $j$-th iteration, the index selected is represented by $k[j]$ such that $x_{k[j]} = b_j$.

        Define $p_j = P(X \in I_j)+ P(X=x_{k[j] + 1})$ and $r_j = \frac{b_j -a_j}{b -a} + \frac{ x_{k[j] + 1} - x_{k[j]}}{b-a}$ for $j \in \{1, \dots, n-1\}$. From $\sum_{j=1}^{n-1} p_j \leq 2P$ and $\sum_{j=1}^{n-1}r_j \leq 2$, we have
        $$
        \sum_{j=1}^{n-1} (p_j + r_j ) \leq 2(1+P).
        $$
        Now, we will get an upper bound of $n$ by estimating a lower bound of $p_j + r_j$.
        For any $j \in \{1, \dots, n-1\}$, since $k[j]$ is the maximum index in  $\{x_{k[j-1]}, \dots, x_{K-1}\}$ such that $B(a_j, x_{k[j]}) \leq \epsilon$, we have
        $$
        B(a_j, x_{k[j] + 1})= (P(X \in I_j) + P(X=x_{k[j] + 1}))\frac{x_{k[j] + 1} -a_j}{4M}  >  \epsilon.
        $$
        From $b_j = x_{k[j]}$ and the definitions of $p_j$ and $r_j$, we have
        \begin{align*}
            \epsilon &< (P(X \in I_j) + P(X=x_{k[j] + 1})) \frac{b_j -a_j + x_{k[j]+1} - x_{k[j]}}{4M} \\
            &= p_j r_j \frac{b-a}{4M}.
        \end{align*}
        Thus, we have
        $$
        \frac{4M\epsilon}{b-a} <  p_jr_j.
        $$
        Following the almost same procedure as in the continuous case, we obtain the following result.
        $$
        n \leq \frac{1+P}{2\sqrt{M}} \sqrt{\frac{b - a}{\epsilon}} + 1.
        $$
    \end{proof}
    \newpage

    \section{Lower bounds of the breakpoints}
    In this section, we derive a lower bound on the number of breakpoints to achieve a given absolute error $\epsilon > 0$ when we use a piecewise linear approximation based on the approach in Section 2. Table \ref{tab:result} summarizes the results from Section 3 and this section. From these results, we see that the upper bounds of the number of breakpoints in our proposed algorithm are close to the lower bounds.
    \begin{table}[ht]
        \centering
        \begin{tabular}{lcc}
            \toprule
            & Continuous                        & Discrete                          \\ \hline
            Lower Bound                                    & $\frac{\sqrt{2}}{4} \approx 0.36$ & $\frac{\sqrt{2}}{2} \approx 0.71$ \\ \hline
            Upper Bound of $B_{\text{exact}}$ or $B_{1/4}$ & $\frac{1}{2}$                     & 1                                 \\
            \bottomrule
        \end{tabular}
        \caption{Coefficients of the Upper and Lower Bounds in Terms of \(\sqrt{\frac{b-a}{\epsilon}}\)}

        \label{tab:result}
    \end{table}

    We first introduce a common setting in discrete or continuous cases. For any given $a$, $b$, $\epsilon$ such that $a < b$ and $\epsilon > 0$, consider a necessary number of intervals, whose approximation error is within $\epsilon$. Define
    \begin{equation}
        \label{def:N}
        N \coloneqq \left\lceil \frac{\sqrt{2}}{2} \sqrt{\frac{W}{\epsilon}} \right\rceil + 1 > \frac{\sqrt{2}}{2} \sqrt{\frac{W}{\epsilon}} ,
    \end{equation}
    where $W = b - a$. Define the following $N$ equally spaced points in $(a, b]$:
    $$
    x_k = a + \frac{W}{N}k \quad (k = 1, \dots N).
    $$
    Roughly speaking, we consider a random variable that takes the value $x_k$ with probability $1/N$ in both the discrete and continuous cases.

    \subsection{Discrete Case}
    Let $X$ be the discrete random variable, whose support is $\{x_1, \dots, x_{N}\}$
    and probability function is
    $$
    p_{X}(x) =  \left \{\begin{array}{ll}
                            \frac{1}{N} & \text{if} \  x = x_k \ (k=1,\dots,N), \\
                            0           & \text{otherwise.}                    \\
    \end{array}
    \right .
    $$
    The distribution of \( X \) is illustrated in Figure~\ref{fig_3}.
    \begin{figure}[ht]
        \centering 
        \includegraphics[width=0.6\textwidth]{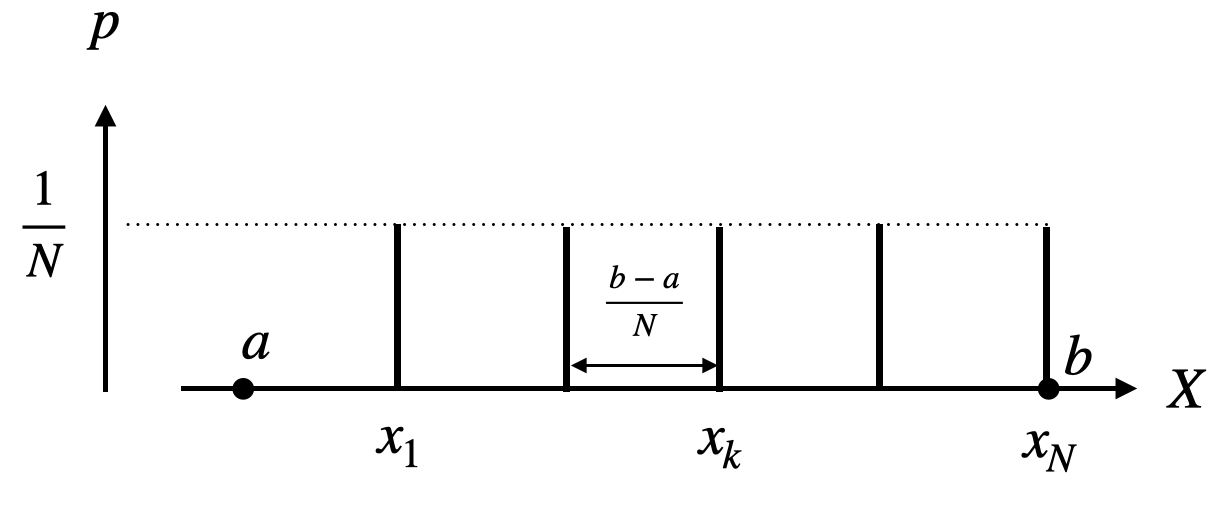} 
        \caption{This is a caption.} 
        \label{fig_3} 
    \end{figure}

    \begin{lemma}
        \label{lemma:low_discrete}
        For any $I=(a', b'] \subseteq (a, b]$,
        $$
        P(X \in I) > \frac{1}{N} \Rightarrow  \Delta_{X}(I) > \epsilon.
        $$
    \end{lemma}

    \begin{proof}
        Let $I \subseteq (a, b]$ be an interval such that $|I \cap \{x_1, \dots, x_N\}| \geq 2$. Without loss of generality, from Lemma \ref{lemma:lower_delta}, for some small positive $o > 0$ and $k \in \{1, \dots, N-1\}$, we assume $I=(x_{k}-o, x_{k+1}]$. From Lemma \ref{lemma:delta}, by defining $\mu = \E[X \mid X \in I] = \frac{x_k + x_{k+1}}{2}$, we have
        $$
        \Delta_{X}(I) = \E_{x_k-o, \mu}[\mu-X] = \E_{x_k,\mu}[\mu- X]
        $$
        where the last equality holds since \(X\) can only take the value \(x_k\) with probability $1/N$ between \(x_k - o\) and  \(\mu\).
        Finally we have
        $$
        \E_{x_k, \mu}[\mu-X] =   \left(\frac{x_k + x_{k+1}}{2} - x_k \right) \cdot \frac{1}{N} = \frac{W}{2N^2} > \epsilon,
        $$
        where we use $x_{k+1} - x_{k} = \frac{W}{N}$ and the last inequality is from \eqref{def:N}.
    \end{proof}

    \begin{theorem}
        Let $\tilde{X}_L$ be a discrete random variable induced by piecewise linear approximation with $\mathcal{I} $ of the random variable $X$. Then,
        $$
        e_{X, \tilde{X}}\leq \epsilon \Rightarrow | \mathcal{I} | > N.
        $$
    \end{theorem}
    \begin{proof}
        Let $\tilde{X}$ be a discrete random variable induced by piecewise linear approximation with $\mathcal{I} $ of a random variable $X$ such that $\max_{s \in (a, b]} |f_{\tilde{X}}(s) -f_X(s)| \leq \epsilon$. For any $I \in \mathcal{I}$, from $\Delta_X(I) \leq \epsilon$ and Lemma \ref{lemma:low_discrete},
        $$
        1 = \sum_{I \in \mathcal{I}} P(X \in I) < \frac{|\mathcal{I}|}{N}.
        $$
    \end{proof}

    \subsection{Continuous Case}
    For some small $w > 0$, let $X$ be the continuous random variable, whose support is $S \coloneqq \cup_{k=1}^{N} (x_k -w, x_k]$
    and probability function is
    \begin{equation}
        p_X(x) =  \left \{\begin{array}{ll}
                              \frac{1}{Nw} & \text{if} \  x \in S, \\
                              0            & \text{otherwise.}    \\
        \end{array}
        \right .
    \end{equation}
    The distribution of \( X \) is illustrated in Figure~\ref{fig_4}.

    \begin{figure}[ht]
        \centering 
        \includegraphics[width=0.6\textwidth]{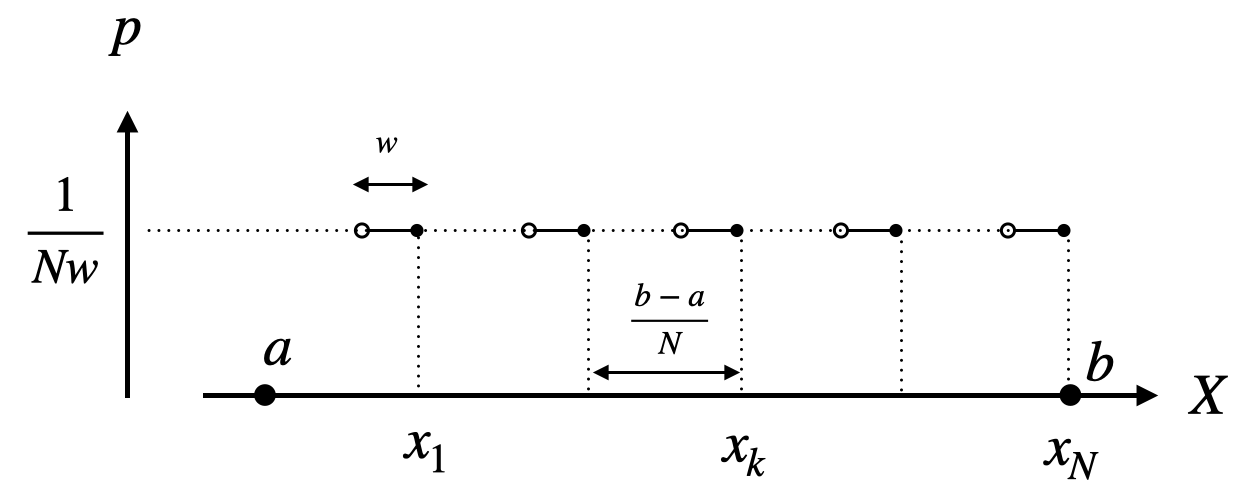} 
        \caption{distribution of $X$} 
        \label{fig_4} 
    \end{figure}

    \begin{lemma}
        \label{lemma:low_cont}
        For $I=(a', b'] \subseteq (a, b]$,
        $$
        P(X \in I) \geq \frac{2}{N} \Rightarrow  \Delta_{X}(I) > \epsilon.
        $$
    \end{lemma}
    \begin{proof}
        From Lemma \ref{lemma:lower_delta}, without loss of generality, we assume that $I=(x_{k}-w, x_{k+1}]$, where $P(X \in I) = \frac{2}{N}$. From Lemma \ref{lemma:delta}, we have
        $$
        \Delta_{X}(I) = \E_{x_k-w, \mu}[\mu-X],
        $$
        where $\mu = \E[X \mid X \in I] = \frac{x_k + x_{k+1}}{2} - \frac{w}{2}$. For each term, we have
        \begin{align*}
            \E_{x_k -w, \mu}[\mu] &=\frac{1}{Nw} \cdot w \cdot \mu = \frac{\mu}{N}\\
            \E_{x_k -w, \mu}[X]&= \E[X \mid X \in (x_k-w, \mu]] \cdot P(X\in (x_k-w, \mu]) \\
            &= \left(x_k - \frac{w}{2} \right) \cdot \frac{w}{Nw}.
        \end{align*}
        Finally, we get
        \begin{align*}
            \E_{x_k-w, \mu}[\mu-X] &= \left( \frac{x_k + x_{k+1}}{2} - \frac{w}{2}- x_k + \frac{w}{2} \right) \cdot \frac{1}{N}\\
            &= \frac{W}{2N^2} > \epsilon
        \end{align*}
        where we use $x_{k+1} - x_{k} = \frac{W}{N}$ and the last inequality is from the definition of $N$ in \eqref{def:N}.
    \end{proof}

    \begin{theorem}
        Let $\tilde{X}$ be a discrete random variable induced by piecewise linear approximation with $\mathcal{I} $ of the random variable $X$. Then,
        $$
        e_{X, \tilde{X}} \leq \epsilon \Rightarrow | \mathcal{I} | > \frac{N}{2}.
        $$
    \end{theorem}
    \begin{proof}
        Let $\tilde{X}$ be a discrete random variable induced by piecewise linear approximation with $\mathcal{I} $ of a random variable $X$ such that $\max_{s \in (a, b]} |f_{\tilde{X}}(s) -f_X(s)| \leq \epsilon$. For any $I \in \mathcal{I}$, from $\Delta_X(I) \leq \epsilon$ and Lemma \ref{lemma:low_cont},
        $$
        1 = \sum_{I \in \mathcal{I}} P(X \in I) < \frac{2|\mathcal{I}|}{N}.
        $$
    \end{proof}

    \newpage

    \section{Numerical Experiments}
    In this section, through numerical experiments, we compare the actual number of breakpoints generated by our proposed algorithms
    with the derived bounds in Section 4 across various distributions\footnote{We have released the proposed algorithm as a Python package at \url{https://github.com/takazawa/piecewise-linearization-first-order-loss},
        where the experimental code is also included.}.

    \subsection{Implementation}
    We implemented our algorithms using Python 3.9 on a Macbook Pro laptop equipped with an Apple M1 Max CPU. We implemented all procedures required for the algorithms and generating distributions using the free scientific computing library, SciPy version 1.9.1 \citep{2020SciPy-NMeth}. For distribution generation, we utilized the distribution classes available in \texttt{scipy.stats}. We performed the calculation of conditional expectations for continuous distributions using the numerical integration function \texttt{quad} in \texttt{scipy.integrate}, with a numerical absolute tolerance set to \(10^{-8}\). Furthermore, we conducted the process of finding the \(b_j\) that satisfies \(B(a_j, b_j) - \epsilon = 0\) in continuous distributions using the \texttt{find\_root} function in \texttt{scipy.optimize}, through binary search, with an absolute tolerance related to \(b_j\) also set to \(10^{-8}\).

    \subsection{Target Distributions}

    Table \ref{tab:dist} summarizes the distributions that are the subject of focus in this study. The \texttt{instance} column lists the instance names, while the \texttt{distribution} column specifies the distribution names. In the \texttt{instance} column, the prefix C or D signifies a continuous or discrete distribution, respectively. In the \texttt{parameter} column, we show the parameters set for each instance. We utilized representative parameters for each distribution; however, for discrete distributions, we set the mean to be approximately 100. Additionally, we represent the approximation intervals in the columns for \(a\) and \(b\). For this study, we set them to be about \(\pm 3\) times the standard deviation from the mean.

    \begin{table}[htb]
        \centering

        {\footnotesize
            \begin{tabular}{lllrr}
                \toprule
                instance & distribution      & parameter           & $a$   & $b$   \\
                \midrule
                C-N1     & Normal            & $\mu=0, \sigma=1$   & -3.0  & 3.0   \\
                C-N2     & Normal            & $\mu=0, \sigma=5$   & -15.0 & 15.0  \\
                C-Exp    & Exponential       & $\lambda=1$         & 0.0   & 4.0   \\
                C-Uni    & Uniform           & $a=0, b=1$          & 0.0   & 1.0   \\
                C-Bet    & Beta              & $\alpha=2, \beta=5$ & 0.0   & 0.8   \\
                C-Gam    & Gamma             & $k=2, \theta=1$     & 0.0   & 6.2   \\
                C-Chi    & Chi-Squared       & $k=3$               & 0.0   & 10.3  \\
                C-Stu    & Student's t       & $\nu=10$            & -3.4  & 3.4   \\
                C-Log    & Logistic          & $\mu=0, s=1$        & -5.4  & 5.4   \\
                C-Lgn    & Lognormal         & $\mu=0, \sigma=1$   & 0.0   & 8.1   \\
                D-Bin    & Binomial          & $n=200, p=0.5$      & 78.0  & 121.0 \\
                D-Poi    & Poisson           & $\lambda=100$       & 70.0  & 130.0 \\
                D-Geo    & Geometric         & $p=0.01$            & 1.0   & 398.0 \\
                D-Neg    & Negative Binomial & $r=100, p=0.5$      & 57.0  & 142.0 \\
                \bottomrule
            \end{tabular}
        }
        \caption{Target Distributions and Intervals} \label{tab:dist}

    \end{table}

    \subsection{Results and Discussion}

    \begin{table}
        \centering

        {\small
            \begin{tabular}{lcrrrrrrrr}
                \toprule
                instance & $\epsilon$ & \multicolumn{5}{c}{the number of intervals} & \multicolumn{3}{c}{error (=$e_{X, \tilde{X}}/ \epsilon$)} \\
                &       & $B_{\text{exact}}$ & $B_{1/8}$ & $B_{1/4}$ & $UB_{1/8}$ & $UB_{1/4}$ & $B_{\text{exact}}$ & $B_{1/4}$ & $B_{1/8}$ \\
                \midrule
                C-N1  & 0.100 & 3                  & 3         & 4         & 3          & 4          & 1.000              & 0.486     & 0.949     \\
                & 0.050 & 4                  & 4         & 6         & 4          & 6          & 1.000              & 0.495     & 0.973     \\
                & 0.010 & 8                  & 8         & 11        & 9          & 13         & 1.000              & 0.499     & 0.996     \\
                \hline
                C-N2  & 0.100 & 6                  & 6         & 8         & 7          & 9          & 1.000              & 0.498     & 0.991     \\
                & 0.050 & 8                  & 8         & 11        & 9          & 13         & 1.000              & 0.499     & 0.996     \\
                & 0.010 & 17                 & 18        & 25        & 20         & 28         & 1.000              & 0.500     & 0.999     \\
                \hline
                C-Exp & 0.100 & 2                  & 3         & 3         & 3          & 4          & 1.000              & 0.490     & 0.955     \\
                & 0.050 & 3                  & 3         & 4         & 4          & 5          & 1.000              & 0.496     & 0.981     \\
                & 0.010 & 7                  & 7         & 9         & 8          & 10         & 1.000              & 0.499     & 0.997     \\
                \hline
                C-Uni & 0.100 & 2                  & 2         & 2         & 2          & 2          & 1.000              & 0.500     & 1.000     \\
                & 0.050 & 2                  & 2         & 3         & 2          & 3          & 1.000              & 0.500     & 1.000     \\
                & 0.010 & 4                  & 4         & 5         & 4          & 6          & 1.000              & 0.500     & 1.000     \\
                \hline
                C-Bet & 0.100 & 1                  & 1         & 2         & 1          & 2          & 0.641              & 0.418     & 0.641     \\
                & 0.050 & 2                  & 2         & 2         & 2          & 2          & 1.000              & 0.425     & 0.837     \\
                & 0.010 & 3                  & 3         & 5         & 4          & 5          & 1.000              & 0.495     & 0.976     \\
                \hline
                C-Gam & 0.100 & 3                  & 3         & 4         & 3          & 4          & 1.000              & 0.491     & 0.939     \\
                & 0.050 & 4                  & 4         & 6         & 4          & 6          & 1.000              & 0.496     & 0.983     \\
                & 0.010 & 8                  & 9         & 12        & 9          & 13         & 1.000              & 0.499     & 0.997     \\
                \hline
                C-Chi & 0.100 & 4                  & 4         & 5         & 4          & 6          & 1.000              & 0.495     & 0.974     \\
                & 0.050 & 5                  & 5         & 7         & 6          & 8          & 1.000              & 0.497     & 0.991     \\
                & 0.010 & 11                 & 11        & 15        & 12         & 16         & 1.000              & 0.499     & 0.998     \\
                \hline
                C-Stu & 0.100 & 3                  & 3         & 4         & 3          & 5          & 1.000              & 0.483     & 0.947     \\
                & 0.050 & 4                  & 4         & 6         & 5          & 6          & 1.000              & 0.494     & 0.967     \\
                & 0.010 & 8                  & 9         & 12        & 10         & 13         & 1.000              & 0.499     & 0.995     \\
                \hline
                C-Log & 0.100 & 4                  & 4         & 5         & 4          & 6          & 1.000              & 0.491     & 0.961     \\
                & 0.050 & 5                  & 5         & 7         & 6          & 8          & 1.000              & 0.496     & 0.983     \\
                & 0.010 & 11                 & 11        & 15        & 12         & 17         & 1.000              & 0.499     & 0.997     \\
                \hline
                C-Lgn & 0.100 & 3                  & 3         & 4         & 4          & 5          & 1.000              & 0.480     & 0.892     \\
                & 0.050 & 4                  & 4         & 6         & 5          & 7          & 1.000              & 0.492     & 0.960     \\
                & 0.010 & 8                  & 9         & 12        & 10         & 15         & 1.000              & 0.498     & 0.993     \\
                \hline
                D-Bin & 0.100 & 7                  & 7         & 12        & 15         & 21         & 0.979              & 0.445     & 0.979     \\
                & 0.050 & 11                 & 12        & 17        & 21         & 30         & 0.922              & 0.497     & 0.889     \\
                & 0.010 & 27                 & 27        & 33        & 47         & 66         & 0.970              & 0.452     & 0.931     \\
                \hline
                D-Poi & 0.100 & 9                  & 9         & 13        & 18         & 25         & 0.978              & 0.440     & 0.969     \\
                & 0.050 & 12                 & 13        & 19        & 25         & 35         & 0.938              & 0.437     & 0.880     \\
                & 0.010 & 33                 & 34        & 42        & 55         & 78         & 0.995              & 0.433     & 0.995     \\
                \hline
                D-Geo & 0.100 & 20                 & 20        & 29        & 44         & 63         & 0.996              & 0.497     & 0.996     \\
                & 0.050 & 29                 & 29        & 41        & 63         & 88         & 0.995              & 0.500     & 0.995     \\
                & 0.010 & 66                 & 68        & 98        & 139        & 197        & 0.983              & 0.496     & 0.995     \\
                \hline
                D-Neg & 0.100 & 10                 & 10        & 15        & 21         & 30         & 0.986              & 0.496     & 0.992     \\
                & 0.050 & 15                 & 15        & 22        & 30         & 42         & 0.992              & 0.451     & 0.992     \\
                & 0.010 & 41                 & 42        & 55        & 66         & 93         & 0.961              & 0.470     & 0.948     \\
                \bottomrule
            \end{tabular}
        }
        \caption{Number of Intervals and Errors}
        \label{tab:all_results}
    \end{table}

    \subsubsection{Details of Table Fields}

    Table \ref{tab:all_results} presents the results obtained in this study. We conducted experiments by varying the allowable error for each instance as \( \epsilon \in \{0.1, 0.05, 0.01\} \) (as shown in the column \( \epsilon \)).
    We conducted experiments on Algorithm \( B \in \{B_{\text{exact}}, B_{1/4}, B_{1/8}\} \) and compared the following 2 items:
    \begin{description}
        \item [\texttt{Number of Interval Column}:] We compared the number of intervals dividing the interval \( (a,b] \) output by each algorithm to the upper bound obtained in this study. These are displayed in the \texttt{intervals} column. The \( B \in \{B_{\text{exact}}, B_{1/4}, B_{1/8}\} \) columns indicate the actual number of intervals output by each algorithm \( B \). The columns for \( B \in \{B_{\text{exact}}, B_{1/4}, B_{1/8}\} \) indicate the actual number of intervals output by each algorithm \( B \).
        More precisely, this refers to \( n \) in Algorithm 1.
        The \( UB \in \{ UB_{1/4}, UB_{1/8}\} \) columns indicate the upper bounds of the breakpoints, which are rounded to integers, shown in Table \ref{tab:breakpoints} in Section 3. Here, \( UB_{1/4} \) and \( UB_{1/8} \) are the upper bounds corresponding to Algorithm \( B_{1/4} \  (B_{\text{exact}}) \) and \( B_{1/8} \), respectively.

        \item [\texttt{Error Column:}] We conducted a comparison between the input \( \varepsilon \) and the error \( e_{X, \tilde{X}} \) associated with \( \tilde{X} \) obtained from the partition output by each algorithm.
        We analytically calculated $e_{X, \tilde{X}}$ based on Corollary \ref{theorem:approx}. Although the results are not included in the paper, we confirmed that the error calculated analytically based on Corollary \ref{theorem:approx} matched the error calculated by SciPy's optimization library. In the table, the error column lists the ratio obtained by dividing the actual error by \( \epsilon \).
    \end{description}

    \subsubsection{Results}
    \begin{enumerate}
        \item Difference between Algorithms:

        \hspace{1em} Although not shown in the table, the computation time for each algorithm was less than 0.2 seconds.

        \hspace{1em} First, we discuss the number of intervals. For any distribution, the number of intervals output by \( B_{\text{exact}} \) and \( B_{1/8} \) differed by only 1-2. The number of intervals for \( B_{1/4} \) was approximately 40\% greater in each instance compared to \( B_{\text{exact}} \) or \( B_{1/8} \).

        \hspace{1em} Next, we discuss the error. For \( B_{\text{exact}} \), the error values for continuous distributions were all 1.0. This is because, except for the last interval, we selected intervals where the error was precisely \( \epsilon \). On the other hand, in the case of discrete distributions, it was not possible to select a point where the error was exactly \( \epsilon \), resulting in slightly smaller values. The errors for \( B_{1/4} \) and \( B_{1/8} \) were slightly less than 0.5 and 1 for most distributions, respectively.

        \item Difference between Actual Values and Upper Bounds:

        \hspace{1em} In continuous distributions, we found that the difference between the actual measured values of intervals and the upper bounds was small. Specifically, the difference between the values of \( B_{\text{exact}} \) or \( B_{1/8} \) and \( UB_{1/8} \), as well as \( B_{1/4} \) and \( UB_{1/4} \), was generally 2 or less. In the case of discrete distributions, the upper bounds were approximately twice the actual measurements.
    \end{enumerate}

    \subsubsection{Discussion}

    We observe that many values in the error column of Algorithm \(B_{1/8}\) are close to 1.
    This suggests that, in many cases, the approximation formula for \eqref{eq:approx} can be valid and a almost tight upper bound of the approximation error.
    Furthermore, considering that the number of breakpoints output by \( B_{1/8} \) hardly differs from \( B_{\text{exact}} \),
    we conclude that \( B_{1/8} \) is an algorithm that outputs near-optimal solutions.

    Regarding continuous distributions, the obtained upper bounds \( UB_{1/4} \) and \( UB_{1/8} \) can be regarded as good approximations to the output values of the corresponding algorithms.
    In other words, \( UB \) can be considered as an expression representing the trade-off between error and the number of breakpoints.
    Therefore, it is conceivable for users to refer to this trade-off when setting an appropriate acceptable error.

    On the other hand, for discrete distributions, we confirm a discrepancy of about twice between the upper bound of intervals and the algorithm output.
    We speculate that this is because, under the worst-case analysis due to the nature of discrete distributions within the proof of Theorem \ref{theorem_discrete}, the upper bound becomes twice as large compared to continuous distributions.
    However, looking at the experimental results, the number of intervals output by the algorithm is close to $\frac{1}{2 \sqrt{2}} \sqrt{\frac{b-a}{\epsilon}}$ which is the upper bound obtained for continuous distributions. Thus, for the discrete case, $\frac{1}{2 \sqrt{2}} \sqrt{\frac{b-a}{\epsilon}}$ may also be useful as a practical guideline for the minimum number of intervals for given error $\epsilon$. Note that this upper bound may not always align with the actual output. If we prepare the number of intervals for the number of possible values of the discrete random variable, the error at that time can become 0. 
     Thus, even when the number of possible values for the random variable is small, if \(b-a\) is large, the actual number of intervals can be much smaller than the bound.

    Lastly, the following table summarizes these observations.
    \begin{table}[htb]
        \centering
        \begin{tabular}{ccc}
            \toprule
            algorithm                               & actual error       & actual intervals                                      \\
            \midrule
            \( B_{\text{exact}} \) and  \( B_{1/8} \) & \( \epsilon \)     & \( \frac{1}{2 \sqrt{2}} \sqrt{\frac{b-a}{\epsilon}}\) \\
            \( B_{1/4} \)                           & \( 0.5 \epsilon \) & \( \frac{1}{2}\sqrt{\frac{b-a}{\epsilon}} \)          \\
            \bottomrule
        \end{tabular}
        \caption{Overview of Experimental Results for Allowable Error $\epsilon$}
        \label{tab:summary}
    \end{table}

    \section{Conclusions}
    This study investigated the trade-off between error and the number of breakpoints when performing a piecewise linear approximation of $f_X(s)=\E[\min(s, X)]$.
    As a result, when conducting piecewise linear approximation based on the method proposed by \cite{rossi2014piecewise}, we obtained \( \frac{1}{2\sqrt{2}}\sqrt{\frac{W}{\epsilon}} \) as an approximately upper bound for the minimum number of breakpoints required to achieve an error less than \( \epsilon \) through theoretical analysis and numerical experiments, where \( W \) is the width of the approximation interval.
    Subsequently, we also proposed efficient algorithms to obtain a piecewise linear approximation with a number of breakpoints close to this bound. These results provide a guideline for the error and number of breakpoints to consider when we use piecewise linear approximation for a general first-order loss function.

    \section*{Acknowledgment}
    This work was partially supported by JSPS KAKENHI Grant Numbers JP21K14368.
    \bibliographystyle{elsarticle-harv}
    \bibliography{cas-refs}

\begin{thebibliography}{18}
\expandafter\ifx\csname natexlab\endcsname\relax\def\natexlab#1{#1}\fi
\providecommand{\url}[1]{\texttt{#1}}
\providecommand{\href}[2]{#2}
\providecommand{\path}[1]{#1}
\providecommand{\DOIprefix}{doi:}
\providecommand{\ArXivprefix}{arXiv:}
\providecommand{\URLprefix}{URL: }
\providecommand{\Pubmedprefix}{pmid:}
\providecommand{\doi}[1]{\href{http://dx.doi.org/#1}{\path{#1}}}
\providecommand{\Pubmed}[1]{\href{pmid:#1}{\path{#1}}}
\providecommand{\bibinfo}[2]{#2}
\ifx\xfnm\relax \def\xfnm[#1]{\unskip,\space#1}\fi
\bibitem[{Barnett and Dragomir(2000)}]{barnett2000some}
\bibinfo{author}{Barnett, N.S.}, \bibinfo{author}{Dragomir, S.S.},
  \bibinfo{year}{2000}.
\newblock \bibinfo{title}{Some inequalities for random variables whose
  probability density functions are absolutely continuous using a pre-chebychev
  inequality}.
\newblock \bibinfo{journal}{RGMIA research report collection}
  \bibinfo{volume}{3}.
\bibitem[{Cox(1971)}]{cox1971algorithm}
\bibinfo{author}{Cox, M.G.}, \bibinfo{year}{1971}.
\newblock \bibinfo{title}{An algorithm for approximating convex functions by
  means by first degree splines}.
\newblock \bibinfo{journal}{The Computer Journal} \bibinfo{volume}{14},
  \bibinfo{pages}{272--275}.
\bibitem[{Gavrilovi{\'c}(1975)}]{gavrilovic1975optimal}
\bibinfo{author}{Gavrilovi{\'c}, M.M.}, \bibinfo{year}{1975}.
\newblock \bibinfo{title}{Optimal approximation of convex curves by functions
  which are piecewise linear}.
\newblock \bibinfo{journal}{Journal of Mathematical Analysis and Applications}
  \bibinfo{volume}{52}, \bibinfo{pages}{260--282}.
\bibitem[{Gutierrez-Alcoba et~al.(2023)Gutierrez-Alcoba, Rossi, Martin-Barragan
  and Embley}]{gutierrez2023stochastic}
\bibinfo{author}{Gutierrez-Alcoba, A.}, \bibinfo{author}{Rossi, R.},
  \bibinfo{author}{Martin-Barragan, B.}, \bibinfo{author}{Embley, T.},
  \bibinfo{year}{2023}.
\newblock \bibinfo{title}{The stochastic inventory routing problem on electric
  roads}.
\newblock \bibinfo{journal}{European Journal of Operational Research}
  \bibinfo{volume}{310}, \bibinfo{pages}{156--167}.
\bibitem[{Imamoto and Tang(2008)}]{imamoto2008recursive}
\bibinfo{author}{Imamoto, A.}, \bibinfo{author}{Tang, B.},
  \bibinfo{year}{2008}.
\newblock \bibinfo{title}{A recursive descent algorithm for finding the optimal
  minimax piecewise linear approximation of convex functions}, in:
  \bibinfo{booktitle}{Advances in Electrical and Electronics Engineering-IAENG
  Special Edition of the World Congress on Engineering and Computer Science
  2008}, \bibinfo{organization}{IEEE}. pp. \bibinfo{pages}{287--293}.
\bibitem[{Kilic and Tunc(2019)}]{kilic2019heuristics}
\bibinfo{author}{Kilic, O.A.}, \bibinfo{author}{Tunc, H.},
  \bibinfo{year}{2019}.
\newblock \bibinfo{title}{Heuristics for the stochastic economic lot sizing
  problem with remanufacturing under backordering costs}.
\newblock \bibinfo{journal}{European Journal of Operational Research}
  \bibinfo{volume}{276}, \bibinfo{pages}{880--892}.
\bibitem[{Liu and Liang(2021)}]{liu2021optimal}
\bibinfo{author}{Liu, B.}, \bibinfo{author}{Liang, Y.}, \bibinfo{year}{2021}.
\newblock \bibinfo{title}{Optimal function approximation with relu neural
  networks}.
\newblock \bibinfo{journal}{Neurocomputing} \bibinfo{volume}{435},
  \bibinfo{pages}{216--227}.
\bibitem[{L{\"o}hndorf(2016)}]{lohndorf2016empirical}
\bibinfo{author}{L{\"o}hndorf, N.}, \bibinfo{year}{2016}.
\newblock \bibinfo{title}{An empirical analysis of scenario generation methods
  for stochastic optimization}.
\newblock \bibinfo{journal}{European Journal of Operational Research}
  \bibinfo{volume}{255}, \bibinfo{pages}{121--132}.
\bibitem[{Ngueveu(2019)}]{ngueveu2019piecewise}
\bibinfo{author}{Ngueveu, S.U.}, \bibinfo{year}{2019}.
\newblock \bibinfo{title}{Piecewise linear bounding of univariate nonlinear
  functions and resulting mixed integer linear programming-based solution
  methods}.
\newblock \bibinfo{journal}{European Journal of Operational Research}
  \bibinfo{volume}{275}, \bibinfo{pages}{1058--1071}.
\bibitem[{Rebennack and Kallrath(2015)}]{rebennack2015continuous}
\bibinfo{author}{Rebennack, S.}, \bibinfo{author}{Kallrath, J.},
  \bibinfo{year}{2015}.
\newblock \bibinfo{title}{Continuous piecewise linear delta-approximations for
  univariate functions: computing minimal breakpoint systems}.
\newblock \bibinfo{journal}{Journal of Optimization Theory and Applications}
  \bibinfo{volume}{167}, \bibinfo{pages}{617--643}.
\bibitem[{Rossi and Hendrix(2014)}]{rossi2014piecewise-conf}
\bibinfo{author}{Rossi, R.}, \bibinfo{author}{Hendrix, E.M.},
  \bibinfo{year}{2014}.
\newblock \bibinfo{title}{Piecewise linearisation of the first order loss
  function for families of arbitrarily distributed random variables}, in:
  \bibinfo{booktitle}{Proceedings of the XII Global Optimization Workshop MAGO
  2014}, pp. \bibinfo{pages}{49--52}.
\bibitem[{Rossi et~al.(2015)Rossi, Kilic and Tarim}]{rossi2015piecewise}
\bibinfo{author}{Rossi, R.}, \bibinfo{author}{Kilic, O.A.},
  \bibinfo{author}{Tarim, S.A.}, \bibinfo{year}{2015}.
\newblock \bibinfo{title}{Piecewise linear approximations for the
  static--dynamic uncertainty strategy in stochastic lot-sizing}.
\newblock \bibinfo{journal}{Omega} \bibinfo{volume}{50},
  \bibinfo{pages}{126--140}.
\bibitem[{Rossi et~al.(2014)Rossi, Tarim, Prestwich and
  Hnich}]{rossi2014piecewise}
\bibinfo{author}{Rossi, R.}, \bibinfo{author}{Tarim, S.A.},
  \bibinfo{author}{Prestwich, S.}, \bibinfo{author}{Hnich, B.},
  \bibinfo{year}{2014}.
\newblock \bibinfo{title}{Piecewise linear lower and upper bounds for the
  standard normal first order loss function}.
\newblock \bibinfo{journal}{Applied Mathematics and Computation}
  \bibinfo{volume}{231}, \bibinfo{pages}{489--502}.
\bibitem[{Shapiro(2003)}]{shapiro2003monte}
\bibinfo{author}{Shapiro, A.}, \bibinfo{year}{2003}.
\newblock \bibinfo{title}{Monte carlo sampling methods}.
\newblock \bibinfo{journal}{Handbooks in operations research and management
  science} \bibinfo{volume}{10}, \bibinfo{pages}{353--425}.
\bibitem[{Snyder and Shen(2019)}]{snyder2019fundamentals}
\bibinfo{author}{Snyder, L.V.}, \bibinfo{author}{Shen, Z.J.M.},
  \bibinfo{year}{2019}.
\newblock \bibinfo{title}{Fundamentals of supply chain theory}.
\newblock \bibinfo{publisher}{John Wiley \& Sons}.
\bibitem[{Tunc et~al.(2018)Tunc, Kilic, Tarim and Rossi}]{tunc2018extended}
\bibinfo{author}{Tunc, H.}, \bibinfo{author}{Kilic, O.A.},
  \bibinfo{author}{Tarim, S.A.}, \bibinfo{author}{Rossi, R.},
  \bibinfo{year}{2018}.
\newblock \bibinfo{title}{An extended mixed-integer programming formulation and
  dynamic cut generation approach for the stochastic lot-sizing problem}.
\newblock \bibinfo{journal}{INFORMS Journal on Computing} \bibinfo{volume}{30},
  \bibinfo{pages}{492--506}.
\bibitem[{Virtanen et~al.(2020)Virtanen, Gommers, Oliphant, Haberland, Reddy,
  Cournapeau, Burovski, Peterson, Weckesser, Bright, {van der Walt}, Brett,
  Wilson, Millman, Mayorov, Nelson, Jones, Kern, Larson, Carey, Polat, Feng,
  Moore, {VanderPlas}, Laxalde, Perktold, Cimrman, Henriksen, Quintero, Harris,
  Archibald, Ribeiro, Pedregosa, {van Mulbregt} and {SciPy 1.0
  Contributors}}]{2020SciPy-NMeth}
\bibinfo{author}{Virtanen, P.}, \bibinfo{author}{Gommers, R.},
  \bibinfo{author}{Oliphant, T.E.}, \bibinfo{author}{Haberland, M.},
  \bibinfo{author}{Reddy, T.}, \bibinfo{author}{Cournapeau, D.},
  \bibinfo{author}{Burovski, E.}, \bibinfo{author}{Peterson, P.},
  \bibinfo{author}{Weckesser, W.}, \bibinfo{author}{Bright, J.},
  \bibinfo{author}{{van der Walt}, S.J.}, \bibinfo{author}{Brett, M.},
  \bibinfo{author}{Wilson, J.}, \bibinfo{author}{Millman, K.J.},
  \bibinfo{author}{Mayorov, N.}, \bibinfo{author}{Nelson, A.R.J.},
  \bibinfo{author}{Jones, E.}, \bibinfo{author}{Kern, R.},
  \bibinfo{author}{Larson, E.}, \bibinfo{author}{Carey, C.J.},
  \bibinfo{author}{Polat, {\.I}.}, \bibinfo{author}{Feng, Y.},
  \bibinfo{author}{Moore, E.W.}, \bibinfo{author}{{VanderPlas}, J.},
  \bibinfo{author}{Laxalde, D.}, \bibinfo{author}{Perktold, J.},
  \bibinfo{author}{Cimrman, R.}, \bibinfo{author}{Henriksen, I.},
  \bibinfo{author}{Quintero, E.A.}, \bibinfo{author}{Harris, C.R.},
  \bibinfo{author}{Archibald, A.M.}, \bibinfo{author}{Ribeiro, A.H.},
  \bibinfo{author}{Pedregosa, F.}, \bibinfo{author}{{van Mulbregt}, P.},
  \bibinfo{author}{{SciPy 1.0 Contributors}}, \bibinfo{year}{2020}.
\newblock \bibinfo{title}{{{SciPy} 1.0: Fundamental Algorithms for Scientific
  Computing in Python}}.
\newblock \bibinfo{journal}{Nature Methods} \bibinfo{volume}{17},
  \bibinfo{pages}{261--272}.
\newblock \DOIprefix\doi{10.1038/s41592-019-0686-2}.
\bibitem[{Xiang et~al.(2023)Xiang, Rossi, Martin-Barragan and
  Tarim}]{xiang2023mathematical}
\bibinfo{author}{Xiang, M.}, \bibinfo{author}{Rossi, R.},
  \bibinfo{author}{Martin-Barragan, B.}, \bibinfo{author}{Tarim, S.A.},
  \bibinfo{year}{2023}.
\newblock \bibinfo{title}{A mathematical programming-based solution method for
  the nonstationary inventory problem under correlated demand}.
\newblock \bibinfo{journal}{European Journal of Operational Research}
  \bibinfo{volume}{304}, \bibinfo{pages}{515--524}.

\end{thebibliography}




    \appendix

    \section{Proofs}

    \begin{lemma}
        \label{lemma:concave}
        $f_X$ is concave on $\mathbb{R}$.
    \end{lemma}
    \begin{proof}
        For any $\alpha \in [0, 1]$ and $s, t \in \mathbb{R}$,
        \begin{align*}
            f_X(\alpha s + (1-\alpha)t) &= \E[\min(\alpha s + (1-\alpha)t, X)] \\
            &\geq  \E [ \alpha \min (s, X)  + (1-\alpha) \min (t, X) ]\\
            &= \alpha \E [\min(s, X)] + (1-\alpha) \E [ \min (t, X) ],
        \end{align*}
        where the inequality is from that $c(s) = \min(s, X)$ is concave.
    \end{proof}

    \begin{lemma}
        \label{lemma:lower_delta}
        Let $X$ be a random variable. Let $I=(a, b]$ and $I' = (a', b']$ such that $P(X \in I) >0$ and $I \subseteq I'$. Then,
        $$
        \Delta_X (I') \geq \Delta_X(I).
        $$
    \end{lemma}
    \begin{proof}
        Let $I=(a, b]$. Define $I^L=(a', b]$ and $I^R = (a, b']$ for any $a' \leq a$ and $b' \geq b$. It suffices to show that $\Delta_X (I^L) \geq \Delta_X (I)$ and $\Delta_X (I^R) \geq \Delta_X (I)$.

        First, we will show that $\Delta_X (I^L) \geq \Delta_X (I)$. From Theorem \ref{theorem:delta},
        $$
        \Delta_X (I^L) - \Delta_X (I) = \E_{a, \mu_L}[\mu_L - X] - \E_{a, \mu_I}[\mu_I - X],
        $$
        where $\mu_L = \E[X \mid X \in I^L]$ and $\mu_I = \E[X \mid X \in I]$.
        Then, we get
        \begin{align*}
            & \E_{a, \mu_L}[\mu_L - X] - \E_{a, \mu_I}[\mu_I - X]\\
            &= \E_{a, \mu_I}[\mu_L - X] + \E_{\mu_I, \mu_L}[\mu_L - X] - \E_{a, \mu_I}[\mu_I - X]\\
            &= \E_{a, \mu_I}[\mu_L - \mu_I] + \E_{\mu_I, \mu_L}[\mu_L - X] \geq 0.
        \end{align*}
        Thus, we have $ \Delta_X (I^L) \geq \Delta_X (I)$.
        We can also show $\Delta_X (I^R) \geq \Delta_X (I)$ similarly. Therefore we have $\Delta_X (I') \geq \Delta_X(I)$.
    \end{proof}

    \subsection{Reduction of General Loss Function to Min Function}\label{sec_reduction}

    We show that the general loss function can be expressed as a sum of $f_X$ and an affine function.
    \begin{lemma}
        For any $A_1, A_2, B_1, B_2 \in \mathbb{R}$,
        $$
        \min(A_1 + B_1, A_2 + B_2) = A_2 + B_1 + \min(A_1 - A_2, B_2 - B_1)
        $$
    \end{lemma}

    \begin{proof}
        For any $A, B \in \mathbb{R}$, the following equation holds.
        $$
        \min (A, B)  = A - \min(A-B, 0).
        $$
        By setting $A=A_1+A_2$ and $B = A_2 + B_2$, we have
        \begin{align*}
            & \min(A_1 + B_1, A_2 +B_2) & \\
            &= A_1 +B_1 - \min(A_1 + B_1 - (A_2 + B_2), 0) & (\because A=A_1+A_2, B = A_2 + B_2)\\
            &= A_1 + B_1 - \min((A_1 - A_2) - (B_2 - B_1), 0) &\\
            &= A_1 +B_1 - (A_1 - A_2 - \min(A_1- A_2, B_2 - B_1)& (\because A = A_1-A_2, B=B_2-B_1)\\
            &= A_2 + B_1 + \min(A_1-A_2, B_2 - B_1).&
        \end{align*}
    \end{proof}

    \begin{theorem}
        For a random variable $X$, define $\ell_X: \mathbb{R} \rightarrow \mathbb{R}$ as,
        \begin{equation}
            \ell_X(s) = \E [ \min(a_1 s + b_1 X + c_1, a_2s+b_2X + c_2)]
        \end{equation}
        where $a_i, b_i, c_i \in \mathbb{R} \ (i=0, 1, 2)$ such that $a_1 \neq a_2$ and $b_1 \neq b_2$.
        Then, $\ell_X$ can be written as
        $$
        \ell_X(s) = As' + B \E[Y] + C + \E[\min(s', Y)]
        $$
        where,$Y=(b_2-b_1)X$ $s'= (a_1-a_2)s + c_1 - c_2$, $A=\frac{a_2}{a_1-a_2}$, $B=\frac{b_1}{b_2-b_1}$, and $C=\frac{a_1c_2- a_2c_1}{a_1-a_2}$.
    \end{theorem}
    \begin{proof}
        From the lemma above, we have
        \begin{align*}
            & \min(a_1 s + b_1 X + c_1, a_2s+b_2X + c_2)\\
            &= \min(\underbrace{a_1 s + c_1}_{A_1} +\underbrace{b_1X}_{B_1}, \underbrace{a_2s+c_2}_{A_2} + \underbrace{b_2X}_{B_2})\\
            &= \min(A_1 + B_1, A_2 + B_2)\\
            &= A_2 + B_1 + \min(A_1 - A_2, B_2 - B_1)\\
            &= a_2s+c_2 + b_1X + \min((a_1-a_2)s + c_1 - c_2, (b_2-b_1)X)\\
            &= As' + BY + C + \min(Y, s').
        \end{align*}
        Thus,
        $$
        \ell_X(s) = As' + B\E[Y] + C + \E[\min(s', Y)].
        $$
    \end{proof}

\end{document}